\documentclass{article}
\usepackage[a4paper,margin=25mm]{geometry}

\usepackage{amsfonts}
\usepackage{amsmath}
\usepackage{amsrefs}
\usepackage{amssymb}
\usepackage{indentfirst}
\usepackage{mathrsfs}

\DeclareMathOperator{\Hilb}{Hilb}
\DeclareMathOperator{\reg}{reg}
\DeclareMathOperator{\initial}{in}
\DeclareMathOperator{\rank}{rank}
\DeclareMathOperator{\gp}{gp}
\DeclareMathOperator{\codim}{codim}

\usepackage{amsthm}
\theoremstyle{definition}
\newtheorem{definition}{Definition}[section]

\newtheorem{remark}[definition]{Remark}
\newtheorem{example}[definition]{Example}

\theoremstyle{plain}
\newtheorem{proposition}[definition]{Proposition}
\newtheorem{theorem}[definition]{Theorem}
\newtheorem{lemma}[definition]{Lemma}
\newtheorem{corollary}[definition]{Corollary}

\title{\bfseries Generators of the initial ideal of simplicial toric ideals}
\author{Ryotaro Hanyu}
\date{}

\begin{document}
\maketitle

\begin{abstract}
We describe a generating set for the initial ideal of simplicial toric ideals with respect to the graded reverse lexicographic order, using representations of elements of affine monoids as sums of irreducible elements. Although the resulting generating set is not necessarily minimal, we demonstrate, through an example, how one can obtain the reduced Gr\"obner basis from it. Moreover, we compare the maximal degree of the Gr\"obner basis and the Castelnuvo-Mumford regularity.  
\end{abstract}


\section{\large Intoroduction}
Let $B$ be an \textit{affine monoid}, that is, $B$ is finitely generated and isomorphic to a submonoid of $\mathbb{Z}^d$ for some $d\geq0$. Let $\mathbb{K}$ be a field and $\mathbb{K}[B]$ a $\mathbb{K}$-vector space with basis $\lbrace t^b\mid b\in B\rbrace$. $\mathbb{K}[B]$ has a natural $\mathbb{K}$-algebra structure with multiplication defined by $t^a\cdot t^b=t^{a+b}$ for $a, b\in B$. $\mathbb{K}[B]$ is called an \textit{affine semigroup ring}. 

An element $b\in B$ is a \textit{unit} if $-b\in B$. An element $b\in B$ is \textit{irreducible} if $b$ is not a unit and that $b=b_1+b_2$ with $b_1, b_2\in B$ implies $b_1$ or $b_2$ is a unit. Every element of $B$ can be expressed as a finite sum of a unit and irreducible elements, and up to differences by units there exist only finitely many irreducible elements in $B$ \cite[Proposition 2.14]{BrGu}. $B$ is said to be \textit{positive} if $0$ is the only unit in $B$. A positive affine monoid $B$ has only finitely many irreducible elements. Let $H$ be the set of all the irreducible elements in $B$. The finite set $H$ is the unique minimal system of generators of $B$, in other words, every system of generators of $B$ contains $H$. The elements of $H$ are called the \textit{Hilbert basis} of $B$, and $H$ is denoted by $\Hilb(B)$. A \textit{positive grading} on $B$ is a monoid homomorphism $\gamma : B\to\mathbb{Z}_{\geq0}$ such that $\gamma(b)=0$ implies $b=0$.

\begin{definition}
An affine monoid $B$ is called \textit{homogeneous} if $B$ is positive and admits  a positive grading $\gamma$ such that $\gamma(a)=1$ for every $a\in\Hilb(B)$.
\end{definition}

A positive grading $\gamma$ induces the $\mathbb{Z}$-grading on $\mathbb{K}[B]$ such that $\deg t^b=\gamma(b)$ for $b\in B$ and $\deg x=0$ for any $x\in\mathbb{K}$. Let $\Hilb(B)=\lbrace a_1,\ldots, a_n\rbrace$ be the Hilbert basis of a homogeneous affine monoid $B$. We will consider the $\mathbb{K}[x_1,\ldots, x_n]$-module structure of $\mathbb{K}[B]$ induced by a degree-preserving surjective homomorphism
\[\pi : S=\mathbb{K}[x_1,\ldots, x_n]\twoheadrightarrow\mathbb{K}[B],\]
defined by $\pi(x_i)=t^{a_i}$, where all the variables $x_i$ of $S$ have degree $1$. Then, $\ker \pi$ is a homogeneous prime ideal, and $\mathbb{K}[B]\simeq S/\ker \pi$. The kernel of $\pi$ is called a \textit{toric ideal}. In the following, we will assume that an affine monoid $B$ is homogeneous.

Let $I$ be a homogeneous ideal of $\mathbb{K}[x_1,\ldots, x_n]$. One way of estimating the complexity of computing Gr\"{o}bner bases of $I$ is to study upper bounds on the degrees of elements in the Gr\"{o}bner basis. If $I$ is generated by homogeneous polynomials of degree at most $d$ and $\mathbb{K}$ is of characteristic zero, then there is an upper bound $(2d)^{(2n+2)^{n+1}}$ for any Gr\"{o}bner basis of $I$ \cite{MoMo}. Such an upper bound is called ``doubly exponential'' (see \cite{HaSe} for more examples of doubly exponential bounds).

One can consider the \textit{Castelnuovo-Mumford regularity} of $I$, denoted by $\reg I$, to provide upper bounds on the degrees of elements in the Gr\"{o}bner basis of $I$. Bayer and Stillman proved that if the choice of coordinates is generic for $I$, then the degrees of the elements in any minimal Gr\"{o}bner basis of $I$ with respect to the graded reverse lexicographic order $\prec$ are at most $\reg I$. In other words, under the assumption, the \textit{initial ideal} $\initial_{\prec}(I)$ of $I$ with respect to $\prec$ is generated in degrees at most $\reg I$ \cite[Corollary 2.5]{BaSt}. We now turn to the case where $I$ is a toric ideal. In particular, we are interested in whether a similar bound holds when $I$ is \textit{simplicial}. This property is known to hold when $S/I$ is generalized Cohen-Macaulay: in this case, $\initial_{\prec}(I)$ is generated in degrees at most $\reg I$ \cite[Lemma 2.1]{HHS}.

In Section 2, we recall some basic definitions and results on simplicial affine semigroup rings that will be used throughout the paper. 

In Section 3, for each element $b$ of a simplicial affine monoid $B$, we define a monomial $m_b$ and use these monomials to describe a generating set for the initial ideal of a simplicial toric ideal. We obtain the following explicit description of the generators.

\begin{theorem}
Let $\ker\pi\subset S$ be a simplicial toric ideal. 
\[
\mathscr{M}_{\initial_{\prec}(\ker\pi)}=\bigcup_{n\in\mathscr{N}_1\cup\mathscr{N}_2}\mathscr{M}_{nS}.
\]
Thus, there exist $n_1,\ldots, n_r\in\mathscr{N}_1\cup\mathscr{N}_2$ such that $\initial_{\prec}(\ker\pi)=(n_1,\ldots, n_r)$. 
\end{theorem}

Here, $\mathscr{M}_{\initial_{\prec}(\ker\pi)}$ and $\mathscr{M}_{nS}$ denote the sets of monomials in $\initial_{\prec}(\ker\pi)$ and $nS$, respectively. In the statement of the theorem, $\mathscr{N}_1$ and $\mathscr{N}_2$ stand for some finite sets of monomials, whose precise descriptions are given in Section 3. Moreover, we demonstrate, through an example, how one can obtain the reduced Gr\"obner basis of $\ker\pi$ from $\mathscr{N}_1\cup\mathscr{N}_2$. 

In Section 4, we consider the decomposition of $\mathbb{K}[B]$ into a finite direct sum of monomial ideals established by Hoa and St\"uckrad\cite{HoSt} (see Subsection 2.3 for details). We prove that if this decomposition is of a particular form, then any minimal Gr\"{o}bner basis of $\ker\pi$ consists of elements of degree at most $r(\mathbb{K}[B])+1$. Here, $r(\mathbb{K}[B])$ denotes the \textit{reduction number} of $\mathbb{K}[B]$, and in general $r(\mathbb{K}[B])+1\leq\reg(\ker\pi)$.

\begin{theorem}
Let $\mathbb{K}[B]\simeq S/\ker\pi$ be a simplicial affine semigroup ring. If every ideal in the decomposition of $\mathbb{K}[B]$ is either the unit ideal or generated by monomials of degree $1$, then $\mathscr{N}_1\cup\mathscr{N}_2$ consists of monomials of degree at most $r(\mathbb{K}[B])+1$. In particular, $\initial_{\prec}(\ker\pi)$ is generated in degrees at most $r(\mathbb{K}[B])+1$. 
\end{theorem}
We remark that the assumption in the statement is satisfied when $\mathbb{K}[B]$ is Buchsbaum (in particular, Cohen-Macaulay) \cite[Proposition 3.1]{BEN}.

In Section 5, we consider the case $\dim\mathbb{K}[B]=3$ and give a necessary condition for $\initial_{\prec}(\ker\pi)$ to have a minimal generator whose degree exceeds $\reg(\ker\pi)$ (see Theorem 5.7 and Subsection 5.3). Furthermore, we give counterexamples to the question of whether $\initial_{\prec}(\ker\pi)$ is generated in degrees at most $\reg(\ker\pi)$.

\section{\large Background}
In Subsection 2.1, we provide some definitions and preliminary results that will be needed only for the proof of the characterization in Subsection 2.2. 
In Subsection 2.2, we give a well-known characterization of simplicial affine semigroup rings.
In Subsection 2.3, we describe a decomposition of simplicial affine semigroup rings due to Hoa and St\"{u}ckrad\cite{HoSt}.

In Section 2, we do not assume that an affine monoid $B$ is homogeneous. Throughout the paper, the $i$-th entry of $x\in\mathbb{R}^n$ is denoted by $x_{[i]}$.

\subsection{Generation of  pointed cones}
Let $X\subset\mathbb{R}^d$ be a subset of $\mathbb{R}^d$. A \textit{cone} generated by $X$ is defined by
\[C(X):=\Bigl\{\sum_{i=1}^kr_ix_i\mid k\geq1,\;r_i\in\mathbb{R}_{\geq0},\;x_i\in X\Bigr\}\subset\mathbb{R}^d.\]
Let $(\mathbb{R}^d)^*$ be the dual vector space. For a finite subset $X\subset\mathbb{R}^d$, we define the \textit{dual cone} of $C(X)$ by
\[C(X)^*:=\lbrace\lambda\in(\mathbb{R}^d)^*\mid\lambda(x)\geq0\quad\text{for all $x\in C(X)$}\rbrace.\]
By Fourier-Motzkin elimination (see \cite[Proposition 2.9.1]{Koc}), we can construct finitely many linear forms $\lambda_1,\ldots, \lambda_t\in(\mathbb{R}^d)^*$ such that 
\[C(X)^*=\mathbb{R}_{\geq0}\lambda_1+\cdots+\mathbb{R}_{\geq0}\lambda_t=C(\lbrace\lambda_1,\ldots, \lambda_t\rbrace).\]
If $C(X)$ is generated by elements in $\mathbb{Q}^d$, then from the construction of $\lambda_i$, we can take each $\lambda_i$ from $(\mathbb{Q}^d)^*$. By the identification $\mathbb{R}^d=(\mathbb{R}^d)^{**}$,
\[C(X)^{**}=H_{\lambda_1}^+\cap\cdots\cap H_{\lambda_t}^+,\]
where $H_{\lambda_i}^+:=\lbrace x\in\mathbb{R}^d\mid\lambda_i(x)\geq0\rbrace$ . By the construction of $\lambda_i$, it follows that the intersection of $H_{\lambda_i}^+$ equals to $C(X)$, then $C(X)=C(X)^{**}$.

If $B$ is an affine monoid, then $C(B)$ is a cone generated by finitely many elements of $B\subset\mathbb{Z}^d$ for some $d\geq0$.
\begin{definition}
$C(B)$ is called \textit{pointed} if $x, -x\in C(B)$ implies $x=0$. 
\end{definition}

Since $C(B)^*$ is also a finitely generated cone, we consider an irredundant representation of $C(B)^*$, 
\[C(B)^*=H_{x_1}^+\cap\cdots\cap H_{x_t}^+,\quad x_i\in(\mathbb{Q}^d)^{**}=\mathbb{Q}^d.\]
If $C(B)$ is pointed, the irredundant representation of $C(B)^*$ is uniquely determined \cite[Theorem 1.6, Proposition 1.19]{BrGu}. By duality, 
\[C(B)=\mathbb{R}_{\geq0}x_1+\cdots\mathbb{R}_{\geq0}x_t,\]
and $\lbrace x_1,\ldots, x_t\rbrace$ is, up to positive scalar multiplication, the unique minimal system of generators of $C(B)$ (cf. \cite[Proposition 1.20]{BrGu}).


If an affine monoid $B$ is positive, then $C(B)$ is pointed \cite[Lemma 3.4.3]{Koc}. For a positive affine monoid $B$, multiplying each generator by a positive scalar if necessary, we can take minimal generators of $C(B)$ from $\Hilb(B)$.
\begin{lemma}
Let $B$ be a positive affine monoid, and $x_1,\ldots, x_t\in\mathbb{Q}^d$ be minimal generators of $C(B)$ obtained as above. Then, for each $x_i$, there exists $m_i\in\mathbb{R}_{\geq0}$ such that $m_ix_i\in\Hilb(B)$.
\end{lemma}
\begin{proof}
Let $x\in\mathbb{Q}^d$ be an element in minimal system of generators of $C(B)$. Since $x\in\mathbb{Q}^d\cap C(B)$, it follows that $x=\sum_{i=1}^kr_ib_i$, for some $r_i\in\mathbb{Q}_{\geq0}$ and $b_i\in B$, by Gaussian elimination. Then, there exists $m\in\mathbb{Q}$ such that $mx\in B$ and $\mathbb{R}_{\geq0}x=\mathbb{R}_{\geq0}mx$. We may assume $\gcd\lbrace mx_{[1]},\ldots, mx_{[d]}\rbrace=1$. Then, we can take a minimal system of generators $b_1,\ldots, b_t$ of $C(B)$ as follows:
\[C(B)=\mathbb{R}_{\geq0}b_1+\cdots+\mathbb{R}_{\geq0}b_t,\]
where $b_i\in B$ and $\gcd\lbrace (b_i)_{[1]},\ldots, (b_i)_{[d]}\rbrace=1$. 

It suffices to show that $b_1$ is irreducible. Assume $b_1=x+y$ with $x, y\in B$. First, we show that $x, y\in\mathbb{R}_{\geq0}b_1$. We can write $x=\sum_{i=1}^tx_ib_i$ and $y=\sum_{i=1}^ty_ib_i$ for some $x_i, y_i\in\mathbb{R}_{\geq0}$. Then, 
\[(1-x_1-y_1)b_1=\sum_{i=2}^t(x_i+y_i)b_i\in B.\]
If $x\notin\mathbb{R}_{\geq0}b_1$, there exists $i\geq2$ such that $x_i>0$ (thus, $x_i+y_i>0$). The minimality of $b_1,\ldots, b_t$ implies that $\sum_{i=2}^t(x_i+y_i)b_i\neq0$. Then, $(1-x_1-y_1)b_1$ is not a unit of $B$ and it follows that $1-x_1-y_1>0$. But,
\[b_1=(1-x_1-y_1)^{-1}\sum_{i=2}^t(x_i+y_i)b_i\in\sum_{i=2}^t\mathbb{R}_{\geq0}b_i,\]
which contradicts the minimality of $b_1,\ldots, b_t$. Thus, $x\in\mathbb{R}_{\geq0}b_1$. Similarly, $y\in\mathbb{R}_{\geq0}b_1$. We can write $x=x_1b_1$ and $y=y_1b_1$ for some $x_1, y_1\in\mathbb{R}_{\geq0}$. Since $x, y, b_1\in B\subset\mathbb{Z}^d$ and $\gcd\lbrace (b_1)_{[1]},\ldots, (b_1)_{[d]}\rbrace=1$, we have $x_1, y_1\in\mathbb{Z}_{\geq0}$. Since $b_1=(x_1+y_1)b_1$, it follows that $(x_1, y_1)=(1, 0)$ or $(x_1, y_1)=(0, 1)$. Thus, $b_1$ is irreducible, that is, $b_1\in\Hilb(B)$. 
\end{proof}

\begin{definition}
An affine monoid $B$ is called \textit{simplicial} if the cone $C(B)$ is generated by linearly independent elements. 
\end{definition}

\begin{corollary}
Let $B$ be a simplicial positive affine monoid. There exists $d\geq0$ such that $B\subset\mathbb{Z}^d$ and $C(B)$ is generated by linearly independent elements $e_1,\ldots, e_d\in\Hilb(B)$.
\end{corollary}
\begin{proof}
The group generated by $B$ is denoted by $\gp(B)$. Set 
\[d=\dim_{\mathbb{Q}}\mathbb{Q}\otimes_{\mathbb{Z}}\gp(B).\]
Since $\gp(B)$ is finitely generated and torsion-free, $\gp(B)\simeq\mathbb{Z}^r$ for some $r\in\mathbb{Z}_{\geq0}$. Then, $r=d$ and we have an embedding $B\hookrightarrow\gp(B)\simeq\mathbb{Z}^d$.

Let $x_1,\ldots, x_t\in\mathbb{Q}^d$ be linearly independent elements generating $C(B)$ with $t\leq d$. Since $x_1,\ldots, x_t$ are minimal generators, $m_1x_1,\ldots, m_tx_t\in\Hilb(B)$ for some $m_1,\ldots, m_t\in\mathbb{R}_{\geq0}$. Set $e_i=m_ix_i$. It is noted that $e_1,\ldots, e_t$ are still linearly independent. Since $B\subset C(B)$ and by Gaussian elimination, we have $\mathbb{Q}B\subset\sum_{i=1}^t\mathbb{Q}e_i$, where
\[\mathbb{Q}B:=\Bigl\{\sum_{i=1}^kr_ib_i\mid k\geq1,\;r_i\in\mathbb{Q},\;b_i\in B\Bigr\},\]
which is isomorphic to $\mathbb{Q}\otimes_{\mathbb{Z}}\gp(B)$. We have $d\leq t$, hence $t=d$.
\end{proof}

\begin{remark}
$\dim_{\mathbb{Q}}\mathbb{Q}\otimes_{\mathbb{Z}}\gp(B)$ is called the \textit{rank} of $B$.
\end{remark}

\subsection{Simplicial affine semigroup rings}
An affine semigroup ring $\mathbb{K}[B]$ is called \textit{simplicial} if $B$ is simplicial. There is the following characterization of simplicial affine semigroup rings.

\begin{proposition}
Let $B$ be an affine monoid and $d=\rank B$. The followings are equivalent:
\begin{enumerate}
  \item $B$ is homogeneous and simplicial.
  \item $B$ is isomorphic to a positive affine monoid $M$ with the Hilbert basis as follows: for some $\alpha\in\mathbb{Z}_{>0}$,
  \[\Hilb(M)=\lbrace a_1,\ldots, a_c, e_1, \ldots, e_d\rbrace\subset\Bigl\{ x\in\mathbb{Z}_{\geq0}^d\mid\sum_{i=1}^dx_{[i]}=\alpha\Bigr\},\]
  where $e_1=(\alpha, 0, \ldots, 0), e_2=(0, \alpha,\ldots, 0),\ldots, e_d=(0,\ldots, 0, \alpha)$.
\end{enumerate}
\end{proposition}
\begin{proof}
$2\Rightarrow 1$: $e_1,\ldots, e_d$ is linearly independent generators of $C(M)$. Moreover, the monoid homomorphism $\gamma: M\to\mathbb{Z}_{\geq0}$ defined by 
\[\gamma(x)=\frac{1}{\alpha}\sum_{i=1}^dx_{[i]},\]
is a positive grading and sends every element of $\Hilb(M)$ to $1$. 

$1\Rightarrow 2$: Let $f_1,\ldots, f_d\in\Hilb(B)$ be linearly independent generators of $C(B)$. Let $\Hilb(B)=\lbrace b_1,\ldots, b_c, f_1,\ldots, f_d\rbrace$. For $\alpha\in\mathbb{Z}_{>0}$, we consider the isomorphism $\phi_{\alpha}: \sum_{i=1}^d\mathbb{R}f_i\to\mathbb{R}^d$ of $\mathbb{R}$-vector spaces, defined by 
\[\phi_{\alpha}(f_i)=(0,\ldots, 0, \alpha, 0,\ldots, 0),\] 
where the $i$-th entry is the only nonzero entry, and equals to $\alpha$. Then, an isomorphism $B\simeq\phi_{\alpha}(B)$ of monoids is induced.  

Since $b_i\in\mathbb{Z}^d\cap\sum_{i=1}^d\mathbb{R}_{\geq0}f_i$, each $b_i$ is contained in $\sum_{i=1}^d\mathbb{Q}_{\geq0}f_i$ by Gaussian elimination. We can write $\phi_{\alpha}(b_i)=(c_i)_{[1]}\phi_{\alpha}(f_1)+\cdots+(c_i)_{[d]}\phi_{\alpha}(f_d)$, where $c_i\in\mathbb{Q}_{\geq0}^d$. Taking a suitable $\alpha$, we may assume $\phi_{\alpha}(b_i)\in\mathbb{Z}_{\geq0}^d$. From the following equalities
\[\alpha\phi_{\alpha}(b_i)=\phi_{\alpha}(b_i)_{[1]}\phi_{\alpha}(f_1)+\cdots+\phi_{\alpha}(b_i)_{[d]}\phi_{\alpha}(f_d),\]
\[\therefore\quad\phi_{\alpha}(\alpha b_i)=\phi_{\alpha}(\phi_{\alpha}(b_i)_{[1]}f_1+\cdots+\phi_{\alpha}(b_i)_{[d]}f_d),\]
we have $\alpha b_i=\phi_{\alpha}(b_i)_{[1]}f_1+\cdots+\phi_{\alpha}(b_i)_{[d]}f_d$. 

There is a monoid homomorphism $\gamma: B\to\mathbb{Z}_{\geq0}$ such that $\gamma(b_i)=1$ and $\gamma(f_j)=1$ for $j=1,\ldots, d$, by the homogeneous condition of $B$. Then, we have $\alpha=\phi_{\alpha}(b_i)_{[1]}+\cdots+\phi_{\alpha}(b_i)_{[d]}$. Set $a_i=\phi_{\alpha}(b_i)$ and $e_i=\phi_{\alpha}(f_i)$. Then, $\phi_{\alpha}(B)$ is a monoid satisfying Condition 2.
\end{proof}

\begin{remark}
If $g:=\gcd\lbrace (a_i)_{[j]}\mid1\leq i\leq c, 1\leq j\leq d\rbrace\neq1$, using $\phi_{\alpha/g}$ through the proof above, we may assume $g=1$ in Condition 2.
\end{remark}

Hereafter, the phrase ``$\mathbb{K}[B]$ is a simplicial affine semigroup ring'' will mean that $B$ is a monoid $M$ in Condition 2. 

Moreover, if $B=M$, we will consider $B$ with a positive grading $\gamma: B\to\mathbb{Z}_{\geq0}$ defined by $\gamma(b)=\frac{1}{\alpha}\sum_{i=1}^db_{[i]}$. Then, $\gamma$ is denoted by $\deg$, and $\deg b$ is called the \textit{degree} of $b$. The degree of $b\in B$ equals to the number of irreducible elements required to express $b$ as a sum of irreducible elements in $B$. Note that we defined $\deg t^b:=\deg b$ for $t^b\in\mathbb{K}[B]$ in Section 1. There is a degree-preserving surjective homomorphism
\[\pi:S=\mathbb{K}[x_1,\ldots, x_c, y_1,\ldots, y_d]\twoheadrightarrow\mathbb{K}[B],\]
defined by $\pi(x_i)=t^{a_i}$ for $i=1,\ldots, c$ and $\pi(y_i)=t^{e_i}$ for $i=1,\ldots, d$. 

In the following, for a simplicial affine semigroup ring $\mathbb{K}[B]$, we will consider the $S$-module structure induced by $\pi$, and call $\ker\pi$ a \textit{simplicial toric ideal}.

Since the images of $y_1,\ldots,y_d$ in $S/\ker\pi$ form a transcendence basis of the fraction field of $S/\ker\pi$ over $\mathbb{K}$, we have $d=\dim\mathbb{K}[B]$. Moreover, $c$ equals to the \textit{hight} of $\ker\pi$, often denoted by $\codim\mathbb{K}[B]$.

\subsection{Decomposition of simplicial affine semigroup rings}
Let $\mathbb{K}[B]$ be a simplicial affine monoid. Then, we may assume that 
\[\Hilb(B)=\lbrace a_1,\ldots, a_c, e_1, \ldots, e_d\rbrace\subset\Bigl\{ x\in\mathbb{Z}_{\geq0}^d\mid\sum_{i=1}^dx_{[i]}=\alpha\Bigr\},\]
for some $\alpha>0$, as in Subsection 2.2. Let $A$ be a monoid generated by $e_1,\ldots, e_d$, that is, $A=\sum_{i=1}^d\mathbb{Z}_{\geq0}e_i=\alpha\mathbb{Z}_{\geq0}^d\subset B$. We consider the set 
\[B_A:=\bigl\{ x\in B\mid x-a\notin B\quad\text{for any $a\in A\setminus\{0\}$}\bigr\}.\]
We note that $x+y\in B_A$ with $x, y\in B$ implies $x, y\in B_A$. Moreover, every element $x$ of $B_A$ is of the form 
\[x=\sum_{i=1}^c\mu_{[i]} a_i\quad\text{with $\mu\in\mathbb{Z}^c$, \;$0\leq\mu_{[i]}<\alpha$}.\]
Then, $B_A$ is a finite set. 
For $x, y\in\gp(B)$, we define
\[x\sim y\iff\ x-y\in\gp(A)=\alpha\mathbb{Z}^d.\]
This is the equivalence relation on $\gp(B)$ modulo its subgroup $\gp(A)$. Let $e$ be the number of equivalence classes in $\gp(B)$, that is, $e:=\#(\gp(B)/\gp(A))$. We restrict the relation to $B_A\subset\gp(B)$, and consider the equivalence classes in $B_A$. The number of equivalence classes in $B_A$ is also $e$ (see \cite[Section 2]{Nit}). Let $\Gamma_1,\ldots, \Gamma_e\subset B_A$ be all the equivalence classes in $B_A$. For each $i=1,\ldots, e$, 
\[h_i:=(\;\min\lbrace \;b_{[1]}\mid b\in\Gamma_i\;\rbrace, \ldots, \min\lbrace \;b_{[d]}\mid b\in\Gamma_i\;\rbrace\;)\in\mathbb{Z}_{\geq0}^d.\]
It is noted that $b-h_i\in A=\alpha\mathbb{Z}_{\geq0}^d$ for any $b\in\Gamma_i$. Then, $h_i\in\gp(B)$. For $i=1,\ldots, e$, we consider
\[I_i:=(t^{b-h_i}\mid b\in\Gamma_i)\mathbb{K}[A],\]
an ideal of $\mathbb{K}[A]$ generated by $t^{b-h_i}$ with $b\in\Gamma_i$.

\begin{proposition}[{\cite[Proposition 2.2]{HoSt}}]
There is an isomorphism of $\mathbb{Z}^d$-graded $\mathbb{K}[A]$-modules,
\[\mathbb{K}[B]\simeq\oplus_{i=1}^eI_i(-h_i).\]
\end{proposition}
\begin{flushright}
$\square$
\end{flushright}

Let $T=\mathbb{K}[y_1,\ldots, y_d]$ be a polynomial ring with $\deg(y_i)=1$. $T$ is isomorphic to $\mathbb{K}[A]$, where $y_i$ corresponds to $t^{e_i}$. Then, each $I_i$ corresponds to a monomial ideal of $T$,
\[\tilde{I_i}:=(y^{\frac{b-h_i}{\alpha}}\mid b\in\Gamma_i)\subset T.\]
Replacing $I_i$ by $\tilde{I_i}$ in the above isomorphism, we obtain 
\[\mathbb{K}[B]\simeq\oplus_{i=1}^e\tilde{I_i}(-h_i).\]
Using the Macaulay2 package MonomialAlgebras, this decomposition can be computed via the command \texttt{decomposeMonomialAlgebra} \cite{BEN}. 

\begin{remark}
The generating set $\{y^{\frac{b-h_i}{\alpha}}\mid b\in\Gamma_i\}$ for $\tilde{I_i}$ is a minimal system of generators. In other words, the number of elements in $\Gamma_i$ is exactly the number of elements in a minimal system of generators of $\tilde{I_i}$.  
\end{remark}

\section{\large Initial ideal of simplicial toric ideals}
In Subsection 3.1, we describe generators of the ideal $\initial_{\prec}(\ker\pi)$ for a simplicial toric ideal $\ker\pi$. In Subsection 3.2, we illustrate, through an example, how one can obtain the reduced Gr\"obner basis of $\ker\pi$.
\subsection{Monomials in the initial ideal of simplicial toric ideals}
Let $\prec$ be the \textit{graded reverse lexicographic order} on $\mathbb{K}[x_1,\ldots, x_c, y_1, \ldots, y_d]$, that is, $\prec$ is a total order defined as follows: $x^{\mu_1}y^{\nu_1}\prec x^{\mu_2}y^{\nu_2}$ with $(\mu_1, \nu_1), (\mu_2, \nu_2)\in\mathbb{Z}_{\geq0}^c\oplus\mathbb{Z}_{\geq0}^d$ if and only if 
\begin{itemize}
\item $\deg x^{\mu_1}y^{\nu_1}<\deg x^{\mu_2}y^{\nu_2}$, or
\item $\deg x^{\mu_1}y^{\nu_1}=\deg x^{\mu_2}y^{\nu_2}$ and the last nonzero entry of $(\mu_2, \nu_2)-(\mu_1, \nu_1)$ is negative.
\end{itemize}
Every polynomial $f\in S=\mathbb{K}[x_1,\ldots, x_c, y_1, \ldots, y_d]$ has the unique \textit{initial term} with respect to $\prec$, denoted by $\initial_{\prec}(f)$. For an ideal $I\subset S$, the \textit{initial ideal} with respect to $\prec$, denoted by $\initial_{\prec}(I)$, is an ideal generated by $\initial_{\prec}(f)$ with $f\in I$. 

Throughout the paper, for subsets $S_1, S_2\subset\mathbb{Z}^d$, 
\[S_1+S_2:=\lbrace a+b\mid a\in S_1, b\in S_2\rbrace.\]
Let $\mathbb{K}[B]$ be a simplicial affine semigroup ring. For $b\in B$,
\[\mathscr{F}_b:=\Bigl\{ x^{\mu}y^{\nu}\mid b=\sum_{i=1}^c\mu_{[i]}a_i+\sum_{i=1}^d\nu_{[i]}e_i\Bigr\}.\]
This is a finite set because every monomial in $\mathscr{F}_{b}$ is of degree $\deg b$. Then, we can define a monomial $m_b$ associated with $b$ as follows:
\[m_b:=\min_{\prec}\lbrace x^{\mu}y^{\nu}\in\mathscr{F}_b\rbrace.\]
Set $\mathscr{M}_B:=\lbrace m_b\mid b\in B\rbrace$.

\begin{proposition}
$\mathscr{M}_{\initial_{\prec}(\ker\pi)}=\mathscr{M}_S\setminus\mathscr{M}_B$.
\end{proposition}
\begin{proof}
Let $n\in\mathscr{M}_{\initial_{\prec}(\ker\pi)}$ and $\pi(n)=t^b$ with $b\in B$. Then, $\mathscr{F}_b$ contains a monomial $m\prec n$; otherwise $n\notin\initial_{\prec}(\ker\pi)$. By definition, $n\notin\mathscr{M}_B$. Thus, $\mathscr{M}_B\subset\mathscr{M}_S\setminus\mathscr{M}_{\initial_{\prec}(\ker\pi)}$. Since $\mathbb{K}[B]$ is a $\mathbb{K}$-vector space with a basis $\lbrace t^b\mid b\in B\rbrace$, there is a bijection induced by $\pi$,
\[\mathscr{M}_B\to\lbrace\text{$\mathbb{K}$-basis of $S/\ker\pi$}\rbrace;\quad m_b\mapsto m_b\mod\ker\pi.\]
Moreover, we have another bijection due to Macaulay's theorem (see \cite[Theorem 15.3]{Eis95}), 
\[\mathscr{M}_S\setminus\mathscr{M}_{\initial_{\prec}(\ker\pi)}\to\lbrace\text{$\mathbb{K}$-basis of $S/\ker\pi$}\rbrace;\quad m\mapsto m\mod\ker\pi.\]
Hence,  $\mathscr{M}_B=\mathscr{M}_S\setminus\mathscr{M}_{\initial_{\prec}(\ker\pi)}$.
\end{proof}

\begin{remark}
Let $a\in A$. The representation $a=\sum_{i=1}^d\nu_{[i]}e_i$ with $\nu\in\mathbb{Z}_{\geq0}^d$ is uniquely determined. Then, $m_a=y^{\nu}\in\mathbb{K}[y_1, \ldots, y_d]$. As an immediate consequence, we have
\[m_{(a+a^{\prime})}=m_am_{a^{\prime}}\quad\text{for all $a, a^{\prime}\in A$.}\]
 
\end{remark}
\begin{proof}
Since $a\in A=\alpha\mathbb{Z}_{\geq0}^d$, the representation $a=\sum_{i=1}^d\nu_{[i]}e_i$ with $\nu\in\mathbb{Z}_{\geq0}^d$ implies that $\nu_{[i]}=a_{[i]}/\alpha$. Thus, the representation is uniquely determined. We show that $y^{\nu}$ is minimal in $\mathscr{F}_a$. Let $x^{\mu}y^{\nu^{\prime}}\in\mathscr{F}_a$. Then, $a=\sum_{i=1}^c\mu_{[i]}a_i+\sum_{i=1}^d\nu^{\prime}_{[i]}e_i$ for some $(\mu, \nu^{\prime})\in\mathbb{Z}_{\geq0}^c\oplus\mathbb{Z}_{\geq0}^d$. Then, $\sum_{i=1}^c\mu_{[i]}a_i=\sum_{i=1}^d(\nu-\nu^{\prime})_{[i]}e_i$. Since $\sum_{i=1}^c\mu_{[i]}a_i\in\mathbb{Z}_{\geq0}^d$, we have $(\nu-\nu^{\prime})_{[i]}\geq0$ for every $i=1, \ldots, d$. If $\mu\neq 0$, then the last nonzero entry of $(\mu, \nu^{\prime})-(0, \nu)$ is negative. Hence $y^{\nu}\prec x^{\mu}y^{\nu^{\prime}}$.
\end{proof}

On the other hand, for $b\in B_A$, since $\mathscr{F}_b\subset\mathbb{K}[x_1, \ldots, x_c]$, the monomial $m_b$ is given by
\[m_{b}=\min_{\prec}\Bigl\{ x^{\mu}\mid b=\sum_{i=1}^c\mu_{[i]}a_i\Bigr\}\in\mathbb{K}[x_1, \ldots, x_c].\] 
For any $b\in B$, as we will see below, $m_b$ can be written in the form $m_{b^{\prime}}m_a$ with $b^{\prime}\in B_A$ and $a\in A$.

\begin{definition}
Let $a, b\in\mathbb{Z}_{\geq0}^d$. The \textit{join} $a\vee b$ of $a$ and $b$ is a vector defined by $(a\vee b)_{[i]}:=\max(a_{[i]}, b_{[i]})$ for $i=1,\ldots, d$.
\end{definition}

Here, we introduce a total order into each equivalence class in $B_A$.
\begin{definition}
Let $b, c\in B_A$ with $b\sim c$, that is, $b-c\in\alpha\mathbb{Z}^d$. 
\[b\prec c\iff\text{the last nonzero entry of $(b-c)$ is negative}.\]
\end{definition} 

Let $\Gamma\in B_A$ be one of the equivalence classes in $B_A$. Since the order $\prec$ on $\Gamma$ is a total order, we can write $\Gamma=\lbrace b_1,\ldots, b_t\rbrace$ with $b_1\prec\cdots\prec b_t$. 

For $1\leq i<j\leq t$, we will use the following notation throughout the paper:
\[n(b_i, b_j):=m_{b_j}m_{(b_i\vee b_j-b_j)}\in\mathscr{F}_{b_1\vee b_2}.\] 
It is noted that we have obtained the descriptions for $m_{b_j}$ and $m_{(b_i\vee b_j-b_j)}$ since $b_j\in B_A$ and $b_i\vee b_j-b_j\in A$.

\begin{remark}
In the notation above, $n(b_i, b_j)\notin\mathscr{M}_B$. Indeed, 
\[m_{b_i}m_{(b_i\vee b_j-b_i)}\prec n(b_i, b_j),\]
and $n(b_i, b_j)-m_{b_i}m_{(b_i\vee b_j-b_i)}\in\ker\pi$.
\end{remark}
\begin{proof}
The kernel of $\pi$ contains $n(b_i, b_j)-m_{b_i}m_{(b_i\vee b_j-b_i)}$, since both $n(b_i, b_j)$ and $m_{b_i}m_{(b_i\vee b_j-b_i)}$ are contained in $\pi^{-1}(t^{b_i\vee b_j})$. Let $\nu_i, \nu_j\in\mathbb{Z}_{\geq0}^d$ such that 
\[b_i\vee b_j-b_i=\sum_{k=1}^d(\nu_i)_{[k]}e_k,\quad b_i\vee b_j-b_j=\sum_{k=1}^d(\nu_j)_{[k]}e_k.\]
Then, $b_i-b_j=\sum_{k=1}^d(\nu_j-\nu_i)_{[k]}e_k$. Since $b_i\prec b_j$, the last nonzero entry of $\nu_j-\nu_i$ is negative. Thus, 
\[m_{b_i}m_{(b_i\vee b_j-b_i)}=m_{b_i}y^{\nu_i}\prec m_{b_j}y^{\nu_j}=n(b_i, b_j).\]
\end{proof}

Let $T=\mathbb{K}[y_1,\ldots, y_d]$. For $\Gamma=\lbrace b_1,\ldots, b_t\rbrace\subset B_A$ with $b_1\prec\cdots\prec b_t$,
\[\mathscr{M}_{\Gamma}:=m_{b_1}\mathscr{M}_T\cup(m_{b_2}\mathscr{M}_T\setminus n(b_1, b_2)\mathscr{M}_T)\cup\cdots\cup (m_{b_t}\mathscr{M}_T\setminus\bigcup_{i=1}^{t-1}n(b_i, b_t)\mathscr{M}_T),\]
where $n\mathscr{M}_T:=\lbrace nm\mid m\in\mathscr{M}_T\rbrace$ for $n\in\mathscr{M}_S$. If $i\neq j$, the sets $m_{b_i}\mathscr{M}_T$ and $m_{b_j}\mathscr{M}_T$ are disjoint since $m_{b_1},\ldots, m_{b_t}$ are distinct monomials in $\mathbb{K}[x_1,\ldots, x_c]$. Then, $\mathscr{M}_{\Gamma}$ is the following disjoint union:
\[\mathscr{M}_{\Gamma}=\bigsqcup_{k=1}^t\left(m_{b_k}\mathscr{M}_T\setminus\bigcup_{i=1}^{k-1}n(b_i, b_k)\mathscr{M}_T\right).\]

\begin{lemma}
Let $\Gamma=\lbrace b_1,\ldots, b_t\rbrace$ with $b_1\prec\cdots\prec b_t$, be one of the equivalence classes in $B_A$. For any $b\in\Gamma+A$, there exists the unique monomial $x^{\mu}y^{\nu}\in\mathscr{M}_{\Gamma}$ such that $b=\sum_{i=1}^c\mu_{[i]}a_i+\sum_{i=1}^d\nu_{[i]}e_i$. 
\end{lemma}
\begin{proof}
 Let $b\in\Gamma+A$, and set $k:=\min\lbrace k^{\prime}\mid b-b_{k^{\prime}}\in A\rbrace$. We show that if $x^{\mu}y^{\nu}\in\mathscr{M}_{\Gamma}$ satisfies $b=\sum_{i=1}^c\mu_{[i]}a_i+\sum_{i=1}^d\nu_{[i]}e_i$, then $x^{\mu}=m_{b_k}$ and $y^{\nu}=m_{(b-b_k)}$. 

It suffices to show that only $x^{\mu}=m_{b_k}$. By construction of $\mathscr{M}_{\Gamma}$, $x^{\mu}$ equals to one of the $m_{b_1},\ldots, m_{b_t}$. By definition of $k$, $x^{\mu}\neq m_{b_i}$ for any $i<k$. Assume that $x^{\mu}=m_{b_i}$ for some $i>k$. Then, 
\[b\in(b_k+A)\cap(b_i+A)=b_k\vee b_i+A,\]
\[\therefore\quad x^{\mu}y^{\nu}=m_{b_i}y^{\nu}\in n(b_k, b_i)\mathscr{M}_T.\]
This contradicts our hypothesis $x^{\mu}y^{\nu}\in\mathscr{M}_\Gamma$. Thus, $x^{\mu}=m_{b_k}$.
\end{proof}

\begin{proposition}
Let $\Gamma=\lbrace b_1,\ldots, b_t\rbrace$ with $b_1\prec\cdots\prec b_t$, be the equivalence class in $B_A$. If $b\in\Gamma+A$, then $m_b=m_{b_k}m_{(b-b_k)}$ where $k:=\min\lbrace k^{\prime}\mid b-b_{k^{\prime}}\in A\rbrace$. In particular, $\mathscr{M}_B=\bigsqcup_{i=1}^e\mathscr{M}_{\Gamma_i}$,
where $\Gamma_1,\ldots, \Gamma_e$ be all the equivalence classes in $B_A$. 
\end{proposition} 
\begin{proof}
Let $b\in B$ and $x^{\mu}y^{\nu}\in\mathscr{F}_b$. Since $B=B_A+A=\bigsqcup_{i=1}^e(\Gamma_i+A)$, there is the unique equivalence class $\Gamma$ such that $b\in\Gamma+A$. Let $\Gamma=\lbrace b_1,\ldots, b_t\rbrace$ with $b_1\prec\cdots\prec b_t$. Then, $x^{\mu}y^{\nu}=m_{b_l}m_{(b-b_l)}$ for some $1\leq l\leq t$. Set $k:=\min\lbrace k^{\prime}\mid b-b_{k^{\prime}}\in A\rbrace$, then $k\leq l$. By the same argument as in Remark 3.5, we have $m_{b_k}m_{(b-b_k)}\preceq x^{\mu}y^{\nu}$. Thus, $m_b=m_{b_k}m_{(b-b_k)}$. Furthermore, from the proof of Lemma 3.6, $\mathscr{M}_{\Gamma}=\lbrace m_b\mid b\in\Gamma+A\rbrace$. Therefore, $\mathscr{M}_B=\bigsqcup_{i=1}^e\mathscr{M}_{\Gamma_i}$.
\end{proof}

Set $\mathscr{M}_1:=\bigcup_{b\in B_A}m_b\mathscr{M}_T$. It is noted that $\mathscr{M}_1\supset\mathscr{M}_B$ and
\[\mathscr{M}_{\initial_{\prec}(\ker\pi)}=\mathscr{M}_S\setminus\mathscr{M}_B=(\mathscr{M}_S\setminus\mathscr{M}_1)\cup(\mathscr{M}_1\setminus\mathscr{M}_B).\]
To describe monomials in $\mathscr{M}_S\setminus\mathscr{M}_1$, we consider the set 
\[\mathscr{N}_1:=\{\text{$n\mid {}^{\exists}x_i$, ${}^{\exists}b\in B_A$ s.t. $n=x_im_b$ and $n\neq m_{(a_i+b)}$}\}.\]
This is a finite subset of $\mathscr{M}_{\mathbb{K}[x_1,\ldots, x_c]}$.

\begin{proposition}
\[\mathscr{M}_S\setminus\mathscr{M}_1=\bigcup_{n\in\mathscr{N}_1}\mathscr{M}_{nS}.\]
\end{proposition}

To prove the proposition, we use the following two lemmas:

\begin{lemma}
Let $n=x_im_b$ with $1\leq i\leq c$ and $b\in B_A$. Then, $n\in\mathscr{N}_1$ if and only if $n\neq m_{b^{\prime}}$ for any $b^{\prime}\in B_A$.
\end{lemma}
\begin{proof}
We have to show the assertion
\[n\neq m_{(a_i+b)} \Leftrightarrow n\neq m_{b^{\prime}}\;\text{for any $b^{\prime}\in B_A$}.\]
Assume $n\neq m_{(a_i+b)}$. Then, since $n\in\pi^{-1}(t^{a_i+b})$, we have $n\neq m_{b^{\prime}}$ for any $b^{\prime}\in B_A$. Conversely, we assume that $n\neq m_{b^{\prime}}$ for any $b^{\prime}\in B_A$. If $a_i+b\in B_A$, it is obvious that $n\neq m_{(a_i+b)}$. If $a_i+b\notin B_A$, we can write $a_i+b=b^{\prime}+a$ with $b^{\prime}\in B_A$ and $a\in A\setminus\{0\}$. Then, $m_{b^{\prime}}m_a\in\pi^{-1}(t^{a_i+b})$. Since $n$ is a monomial in $\mathbb{K}[x_1, \ldots, x_c]$, whereas $m_a$ is a monomial in $\mathbb{K}[y_1, \ldots, y_d]$, we have $m_{b^{\prime}}m_a\prec n$. Thus $n\neq m_{(a_i+b)}$.
\end{proof}

\begin{lemma}
 Let $b\in B_A$, $m\in\mathscr{M}_S$. If $m_b$ is divisible by $m$, then $m=m_{b^{\prime}}$ for some $b^{\prime}\in B_A$.
 \end{lemma}
\begin{proof}
Since $m_b\in\mathscr{M}_{\mathbb{K}[x_1,\ldots, x_c]}$, we can write $m_b/m=x^{\mu}$ and $m=x^{\mu^{\prime}}$ with $\mu, \mu^{\prime}\in\mathbb{Z}_{\geq0}^c$. Set $b^{\prime}=\sum_{i=1}^c\mu^{\prime}_{[i]}a_i$. Then $b=b^{\prime}+\sum_{i=1}^c\mu_{[i]}a_i$, and $b^{\prime}\in B_A$ since $b\in B_A$. Assume $m_{b^{\prime}}\neq x^{\mu^{\prime}}$. Then, there exists $\mu^{\prime\prime}\in\mathbb{Z}_{\geq0}^c$ such that $x^{\mu^{\prime\prime}}\prec x^{\mu^{\prime}}$ and $\sum_{i=1}^c(\mu^{\prime\prime})_{[i]}a_i=b^{\prime}$. Then, we have 
\[x^{\mu^{\prime\prime}}x^{\mu}\prec x^{\mu^{\prime}}x^{\mu}=m_b\quad\text{and}\quad\sum_{i=1}^c(\mu^{\prime\prime}+\mu)_{[i]}a_i=b,\] 
which contradicts the definition of $m_b$. Hence $x^{\mu^{\prime}}=m_{b^{\prime}}$ with $b^{\prime}\in B_A$.
\end{proof}

\begin{proof}[Proof of Proposition 3.8]
Let $n\in\mathscr{N}_1$. Assume that $\mathscr{M}_{nS}\cap\mathscr{M}_1$ is not empty. Then, there exist $\mu\in\mathbb{Z}_{\geq0}^c$, $\nu, \nu^{\prime}\in\mathbb{Z}_{\geq0}^d$ and $b\in B_A$ such that $nx^{\mu}y^{\nu}=m_by^{\nu^{\prime}}$. We have $nx^{\mu}=m_b$, and by Lemma 3.10, $n=m_{b^{\prime}}$ for some $b^{\prime}\in B_A$. By Lemma 3.9, this contradicts that $n\in\mathscr{N}_1$. Thus, $\mathscr{M}_{nS}$ and $\mathscr{M}_1$ are disjoint for every $n\in\mathscr{N}_1$, which implies $\mathscr{M}_S\setminus\mathscr{M}_1\supset\bigcup_{n\in\mathscr{N}_1}\mathscr{M}_{nS}$.

To prove the reverse inclusion, we prove the following claim:
\[x^{\mu}y^{\nu}\notin\bigcup_{n\in\mathscr{N}_1}\mathscr{M}_{nS}\implies x^{\mu}=m_b\text{ for some $b\in B_A$ (then $x^{\mu}y^{\nu}\in\mathscr{M}_1$)}.\]
Let $X_{\mu}$ be the set of monomials,
\[X_{\mu}:=\lbrace m_b\mid b\in B_A,\text{ and $m_b$ divides $x^{\mu}$}\rbrace.\] 

If $X_{\mu}$ is empty, $x_i=m_{a_i}$ does not divide $x^{\mu}$ for any $i=1,\ldots, c$, and then $x^{\mu}=1=m_0$, a monomial associated with $0\in B_A$. 

If $X_{\mu}$ is not empty, let $m_b\in X_{\mu}$ be a monomial with the maximal total degree. We will show that $m_b=x^{\mu}$. 
Assume that $m_b\neq x^{\mu}$. Then, $x^{\mu}$ is divisible by $m:=x_im_b$ for some $x_i$, since $m_b$ divides $x^{\mu}$. From the choice of $m_b$, we know that $m\notin X_{\mu}$. Hence $m\neq m_{b^{\prime}}$ for any $b^{\prime}\in B_A$, and $m\in\mathscr{N}_1$. Then, we have $x^{\mu}y^{\nu}\in\mathscr{M}_{mS}$, which contradicts that $x^{\mu}y^{\nu}\notin\bigcup_{n\in\mathscr{N}_1}\mathscr{M}_{nS}$. 
\end{proof}

Now we have $\mathscr{M}_{\initial_{\prec}(\ker\pi)}=(\bigcup_{n\in\mathscr{N}_1}\mathscr{M}_{nS})\cup(\mathscr{M}_1\setminus\mathscr{M}_B)$. Next, we turn our attention to the set $\mathscr{M}_1\setminus\mathscr{M}_B$.

For an equivalence class $\Gamma=\lbrace b_1,\ldots, b_t\rbrace\subset B_A$ with $b_1\prec\cdots\prec b_t$,
\[\mathscr{N}_{\Gamma}:=\lbrace n(b_i, b_j)\mid 1\leq i<j\leq t\rbrace.\]
If $\#\Gamma=1$, we define $\mathscr{N}_{\Gamma}:=\phi$. From the definition of $\mathscr{M}_{\Gamma}$, we have
\[\bigcup_{i=1}^tm_{b_i}\mathscr{M}_T=\mathscr{M}_{\Gamma}\cup(\bigcup_{n\in\mathscr{N}_{\Gamma}}n\mathscr{M}_T).\]
Let $\Gamma_1,\ldots, \Gamma_e$ be all the equivalence classes on $B_A$. We define
\[\mathscr{N}_2:=\bigcup_{i=1}^e\mathscr{N}_{\Gamma_i}.\]
Note that $\mathscr{N}_2$ is a finite set, and $\mathscr{M}_1=\mathscr{M}_B\cup(\bigcup_{n\in\mathscr{N}_2}n\mathscr{M}_T)$.

\begin{theorem}
\[\mathscr{M}_{\initial_{\prec}(\ker\pi)}=\bigcup_{n\in\mathscr{N}_1\cup\mathscr{N}_2}\mathscr{M}_{nS}.\]
Thus, there exist $n_1,\ldots, n_r\in\mathscr{N}_1\cup\mathscr{N}_2$ such that $\initial_{\prec}(\ker\pi)=(n_1,\ldots, n_r)$.
\end{theorem}
\begin{proof}
We have $\mathscr{M}_{\initial_{\prec}(\ker\pi)}=(\bigcup_{n\in\mathscr{N}_1}\mathscr{M}_{nS})\cup(\mathscr{M}_1\setminus\mathscr{M}_B)$, and
\[\mathscr{M}_1\subset\mathscr{M}_B\cup(\bigcup_{n\in\mathscr{N}_2}\mathscr{M}_{nS}).\]
Thus, $\mathscr{M}_{\initial_{\prec}(\ker\pi)}\subset\bigcup_{n\in\mathscr{N}_1\cup\mathscr{N}_2}\mathscr{M}_{nS}$.

To prove the reverse inclusion, it suffices to show $n\in\initial_{\prec}(\ker\pi)$ for any $n\in\mathscr{N}_1\cup\mathscr{N}_2$. We have seen $\mathscr{N}_1\subset\initial_{\prec}(\ker\pi)$. Let $n\in\mathscr{N}_2$. We can write $n=n(b_i, b_j)$ with $b_i\prec b_j$ in an equivalence class $\Gamma\subset B_A$. Then,
\[n(b_i, b_j)-m_{b_i}m_{(b_i\vee b_j-b_i)}\in\ker\pi\quad\text{and}\quad m_{b_i}m_{(b_i\vee b_j-b_i)}\prec n(b_i, b_j),\]
see Remark 3.5. Thus, $n(b_i, b_j)\in\initial_{\prec}(\ker\pi)$.
\end{proof}

\begin{example}
Let $B$ be a simplicial affine monoid with
\[\Hilb(B)=\lbrace\;\underbrace{(11, 1)}_{a_1},\;\underbrace{(9, 3)}_{a_2},\;\underbrace{(4, 8)}_{a_3},\;\underbrace{(1, 11)}_{a_4},\;\underbrace{(12, 0)}_{e_1},\;\underbrace{(0, 12)}_{e_2}\;\rbrace,\]
and $A=\mathbb{Z}_{\geq0}e_1+\mathbb{Z}_{\geq0}e_2=12\cdot\mathbb{Z}_{\geq0}^2$.

A monomial associated with $(2, 22)$ is $m_{(2, 22)}=x_4^2$. Since $a_1+2a_4=(13, 23)$ is not in $B_A$, we have $x_1x_4^2\in\mathscr{N}_1$. Note that $x_1x_4$ is also in $\mathscr{N}_1$, and thus $x_1x_4^2\in\mathscr{N}_1$ is unnecessary for generating $\initial_{\prec}(\ker\pi)$.
To illustrate another example of a monomial in $\mathscr{N}_1$, we consider 
\[(19, 17)=2a_2+a_4=a_1+2a_3\in B_A.\]
A monomial associated with $(19, 17)$ is $m_{(19, 17)}=x_2^2x_4$. On the other hand, $x_1x_3^2\in\mathscr{N}_1$ since $x_3^2=m_{(8, 16)}$ with $(8, 16)\in B_A$.

In $B_A$, there exist an equivalence class 
\[\Gamma=\lbrace\;\underbrace{(18, 6)}_{2a_2},\;\underbrace{(6, 30)}_{a_3+2a_4}\;\rbrace\subset B_A,\]
where $2a_2$ and $a_3+2a_4$ are representations of $(18, 6)$ and $(6, 30)$ respectively, which correspond to $m_{(18, 6)}$, $m_{(6, 30)}$. It is noted that $(18, 6)\prec (6, 30)$. Set $b_1:=(18, 6)$, $b_2:=(6, 30)$. Then, $x_3x_4^2y_1=n(b_1, b_2)\in\mathscr{N}_2$. 
\end{example}
As shown in the example above, $\mathscr{N}:=\mathscr{N}_1\cup\mathscr{N}_2$ contains redundant elements for generating $\initial_{\prec}(\ker\pi)$. If $n^{\prime}\in\mathscr{N}$ is divisible by $n\in\mathscr{N}$, remove $n^{\prime}$. Repeating this process, we obtain a subset $\mathscr{N}_0\subset\mathscr{N}$ such that $n_1\neq n_2$ with $n_1, n_2\in\mathscr{N}_0$ implies that $n_1$ is neither divisible by $n_2$ nor $n_2$ divisible by $n_1$. Then $\mathscr{N}_0$ is a minimal system of generators of $\initial_{\prec}(\ker\pi)$.

\begin{definition}
Let $I$ be an ideal of $S$. A finite subset $G\subset I$ is a \textit{Gr\"obner basis} of $I$ with respect to $\prec$ if $\initial_{\prec}(I)$ is generated by $\lbrace\initial_{\prec}(g)\mid g\in G\rbrace$. 
\end{definition}
 A Gr\"obner basis $G$ with respect to $\prec$ is called \textit{reduced} if for any $g\in G$, no term of $g^{\prime}$ with $g^{\prime}\in G\setminus\lbrace g\rbrace$ is divisible by $\initial_{\prec}(g)$. It is known that the reduced Gr\"obner basis is unique for an ideal and a term order, provided that the coefficient of the initial term of $g$ is $1$ for each $g\in G$.

Let $n=x^{\mu}y^{\nu}\in\mathscr{N}$ with $b=\sum_{i=1}^c\mu_{[i]}a_i+\sum_{i=1}^d\nu_{[i]}e_i$. To describe a Gr\"obner basis of $\ker\pi$, we consider a binomial
\[g_n:=n-m_b\in\ker\pi.\]
 
\begin{proposition}
In the notation above, $G=\lbrace g_n\mid n\in\mathscr{N}_0\rbrace$ is the reduced Gr\"obner basis of $\ker\pi$ with respect to the graded reverse lexicographic order $\prec$.
\end{proposition}
\begin{proof}
By definition, $G$ is a Gr\"obner basis of $\ker\pi$ with respect to $\prec$, and the coefficient of each $n=\initial_{\prec}(g_n)$ is $1$. Let $n\in\mathscr{N}_0$ and $n^{\prime}\in\mathscr{N}_0\setminus\lbrace n\rbrace$. A binomial $g_{n^{\prime}}$ is of the form $n^{\prime}-m$ with $m\in\mathscr{M}_B$. By construction of $\mathscr{N}_0$, $n^{\prime}$ is not divisible by $n$. Moreover, $m$ is not divisible by $n$ since $\mathscr{M}_{\initial_{\prec}(\ker\pi)}=\mathscr{M}_S\setminus\mathscr{M}_B$.
\end{proof}

\begin{remark}
Proposition 3.14 means that a set $\mathscr{N}_0\subset\mathscr{N}$ obtained by removing redundant elements for generating $\initial_{\prec}(\ker\pi)$ is exactly the set of initial terms of the reduced Gr\"obner basis of $\ker\pi$. The second term of $g_n$, that is $n-g_n$, is determined by Proposition 3.7.
\end{remark}

\begin{remark}
Let $n\in\mathscr{N}_2\cap\mathscr{N}_0$. Then, there exist $b_i\prec b_j$ in an equivalence class $\Gamma\subset B_A$ such that
\[g_n=m_{b_j}m_{(b_i\vee b_j-b_j)}-m_{b_i}m_{(b_i\vee b_j-b_i)}.\]
In other words, we can write $n=n(b_i, b_j)$ such that $b_i=\min_{\prec}\{b\in\Gamma\mid b_i\vee b_j-b\in A\}$. This notation will be used in Section 5.
\end{remark}
\begin{proof}
Let $n\in\mathscr{N}_2\cap\mathscr{N}_0$. We can write $n=n(b_i, b_j)$ with $b_i\prec b_j$ in an equivalence class $\Gamma\subset B_A$. Let $b^{\prime}=\min_{\prec}\{b\in\Gamma\mid b_i\vee b_j-b\in A\}$. We show that $n=n(b^{\prime}, b_j)$. Since $b_1\vee b_j-b^{\prime}\in A$, we have $b_i\vee b_j-b^{\prime}\vee b_j\in A$. Then, 
\[(b_i\vee b_j-b_j)-(b^{\prime}\vee b_j-b_j)=b_i\vee b_j-b^{\prime}\vee b_j\in A.\]
This implies that $m_{(b_i\vee b_j-b_j)}$ is divisible by $m_{(b^{\prime}\vee b_j-b_j)}$. Then, $n=m_{b_j}m_{(b-b_j)}$ is divisible by $n(b^{\prime}, b_j)=m_{b_j}m_{(b^{\prime}\vee b_j-b_j)}\in\mathscr{N}_2$. Since $n\in\mathscr{N}_0$, we have $n(b_i, b_j)=n(b^{\prime}, b_j)$. Thus, we may assume that $n=n(b_i, b_j)$ with $b_i=\min_{\prec}\{b\in\Gamma\mid b_i\vee b_j-b\in A\}$. By Proposition 3.7, we have $m_{b_i\vee b_j}=m_{b_i}m_{(b_i\vee b_j-b_i)}$.
\end{proof}

\subsection{Computing a reduced Gr\"obner basis}
Let $B$ be a simplicial affine monoid with
\[\Hilb(B)=\lbrace\;\underbrace{(0, 1, 3)}_{a_1},\;\underbrace{(2, 0, 2)}_{a_2},\;\underbrace{(3, 1, 0)}_{a_3},\;\underbrace{(4, 0, 0)}_{e_1},\;\underbrace{(0, 4, 0)}_{e_2},\;\underbrace{(0, 0, 4)}_{e_3}\;\rbrace.\]
To obtain $\mathscr{N}_0$, we list the representations of elements $b\in B_A$ corresponding to $m_b$, where the monomials $m_b$ are ordered increasingly with respect to $\prec$. 

The set of elements in $B_A$ of degree $t$ is denoted by $(B_A)_t$. It is noted that every element in $(B_A)_{t+1}$ is of the form $a_i+b$ with $b\in(B_A)_t$. First, $0\in (B_A)_0$ and the elements in $(B_A)_1$ are listed as follows: 
\[0,\;a_3,\;a_2,\;a_1.\]
To list the elements in $(B_A)_2$, we examine the list 
\[a_3+a_3,\;a_2+a_3,\;a_1+a_3,\;a_2+a_2,\;a_1+a_2,\;a_1+a_1,\]
in the order, and check whether the corresponding monomial of each representation equals to $m_b$ with $b\in B_A$: For example,
\begin{itemize}
\item $x^{2a_3}:=x_3^2=m_{(6, 2, 0)}$, where $2a_3=(6, 2, 0)\in B_A$.
\item $x^{2a_2}:=x_2^2\neq m_{b}$ for any $b\in B_A$ since $2a_2=(4, 0, 4)\notin B_A$. 
\end{itemize}
The representations corresponding to $m_b$ for some $b\in B_A$ are collected in their original order, yielding the list of $(B_A)_2$:
\[2a_3,\;a_2+a_3,\;a_1+a_3,\;a_1+a_2,\;2a_1.\]
The representations other than above are collected as well, and the resulting set denoted by $(N_1^{\prime})_2$: $(N_1^{\prime})_2=\lbrace 2a_2\rbrace$ and $x^{2a_2}:=x_2^2\in\mathscr{N}_1$.

Similarly, to list the elements in $(B_A)_3$, we examine the list  
\[a_3+(2a_3),\;a_2+(2a_3),\;a_1+(2a_3),\;a_2+(a_2+a_3),\;a_1+(a_2+a_3),\] 
\[a_1+(a_1+a_3),\;a_1+(a_1+a_2),\;a_1+(2a_1).\]
Then, the elements in $(B_A)_3$ are listed as follows:
\[3a_3,\;a_2+2a_3,\;a_1+2a_3,\;a_1+a_2+a_3,\;2a_1+a_3,\;2a_1+a_2,\;3a_1,\]
 and $(N_1^{\prime})_3=\lbrace 2a_2+a_3\rbrace$. It is noted that $x^{(2a_2+a_3)}:=x_2^2x_3\in\mathscr{N}_1$.

To describe this procedure in general, we formalize the operation that produces the list of $(B_A)_{t+1}$ from that of $(B_A)_t$ as follows: 

\vspace{0.5em}
\noindent\emph{\textbf{Procedure.}}
Let $\sigma_1,\ldots,\sigma_r$ be the ordered list of $(B_A)_t$. Let $i_\sigma$ denotes the smallest index of $a_i$ that appears in a representation $\sigma\in\lbrace\sigma_1,\ldots,\sigma_r\rbrace$. Then, we examine the following elements, in this order, 
\[a_{i_{\sigma_1}}+\sigma_1,\;a_{(i_{\sigma_1}-1)}+\sigma_1,\;\ldots,\;a_1+\sigma_1,\]
\[a_{i_{\sigma_2}}+\sigma_2,\;a_{(i_{\sigma_2}-1)}+\sigma_2,\;\ldots,\;a_1+\sigma_2,\]
\[\ldots,\]
\[a_{i_{\sigma_r}}+\sigma_r,\;a_{(i_{\sigma_r}-1)}+\sigma_r, \;\ldots,\; a_1+\sigma_r,\]
and check whether each representation $\tau$ satisfies the following two conditions:
\begin{enumerate}
\item The value of $\tau$ does not coincide with the value of any representation that appears before $\tau$.
\item The value of $\tau$ belongs to $B_A$.
\end{enumerate}
The representations satisfying the conditions are collected in their original order, yielding the list of $(B_A)_{t+1}$. Simultaneously, the representations that fail to satisfy the conditions are collected, and let $(N_1^{\prime})_{t+1}$ be the resulting set.  

\begin{proof}[Justification of Procedure]
We observe that if $\tau$ does not satisfy both of the above conditions, then the corresponding monomial $x^{\tau}$ does not equal to $m_b$ for any $b\in B_A$. If $\tau$ does not satisfy Condition 2, it is obvious. If $\tau$ does not satisfy Condition 1 but does satisfy Condition 2, there is a representation $\tau^{\prime}$ appearing before $\tau$ such that $\tau$ and $\tau^{\prime}$ have the same value. It implies that $x^{\tau^{\prime}}\prec x^{\tau}$ and $x^{\tau^{\prime}}, x^{\tau}\in\mathscr{F}_b$ for some $b\in B_A$. Then, $x^{\tau}\neq m_{b}$. 

Next, we observe that, to obtain the list of $(B_A)_{t+1}$, there is no need to check whether a representation of the form $a_i+\sigma$ with $i>i_{\sigma}$ satisfies the conditions. Let $\sigma^{\prime}=\sigma-a_{i_{\sigma}}+a_i$, then $a_i+\sigma=a_{i_{\sigma}}+\sigma^{\prime}$ and $x^{\sigma^{\prime}}\prec x^{\sigma}$. It implies that the representation $a_i+\sigma$ has already been examined as $a_{i_{\sigma}}+\sigma^{\prime}$ if $\sigma^{\prime}$ is in the list of $(B_A)_t$ ; it is noted that if $\sigma^{\prime}$ is not in the list of $(B_A)_t$, then $a_i+\sigma$ is not in the list of $(B_A)_{t+1}$, since by Lemma 3.10, $x_ix^{\sigma}=x_{i_{\sigma}}x^{\sigma^{\prime}}\neq m_{b}$ for any $b\in B_A$.
\end{proof}

\begin{remark}
In Procedure, rather than merely checking whether $\sigma$ belongs to $B_A$, we check whether the corresponding monomial $x^{\sigma}$ coincides with $m_b$ for some element $b\in B_A$. This prevents redundant enumeration of elements in $B_A$, and more importantly, allows us to distinguish representations corresponding to monomials in $\mathscr{N}_1$ from those corresponding to $m_b$ with $b\in B_A$ (for example, $a_1+2a_3$ and $2a_2+a_4$ in Example 3.12). 
\end{remark}

Returning to the computation, the elements in $(B_A)_4$ are listed as follows: 
\[a_2+3a_3,\;a_1+3a_3,\;a_1+a_2+2a_3,\;2a_1+a_2+a_3,\;3a_1+a_3,\;3a_1+a_2.\]
Simultaneously, we collect the representations not satisfying the condition: 
\[(N_1^{\prime})_4=\lbrace\;4a_3,\;2a_2+2a_3,\;2a_1+2a_3,\;4a_1\;\rbrace.\] 
The corresponding monomials $x_3^4$, $x_2^2x_3^2$, $x_1^2x_3^2$, $x_1^4$ are in $\mathscr{N}_1$. The elements in $(B_A)_5$ are listed as follows:
\[a_1+a_2+3a_3,\;3a_1+a_2+a_3,\]
and the representations not satisfying the condition are
\[(N_1^{\prime})_5=\lbrace\;2a_2+3a_3,\;2a_1+3a_3,\;2a_1+a_2+2a_3,\;4a_1+a_3\;\rbrace.\]
Thus, $x_2^2x_3^3$, $x_1^2x_3^3$, $x_1^2x_2x_3^2$, $x_1^4x_3$ are in $\mathscr{N}_1$. This completes the list of $B_A$ because there is no element in $(B_A)_6$:
\[(N_1^{\prime})_6=\lbrace\;2a_1+a_2+3a_3,\;4a_1+a_2+a_3\;\rbrace\quad\text{and}\quad x_1^2x_2x_3^3,\;x_1^4x_2x_3\in\mathscr{N}_1.\]
Summarizing the above computation, the elements of $B_A$ are listed as follows:
\begin{enumerate}
\item[$(B_A)_0$ :] $0$.
\item[$(B_A)_1$ :] $a_3,\;a_2,\;a_1$.
\item[$(B_A)_2$ :] $2a_3,\;a_2+a_3,\;a_1+a_3,\;a_1+a_2,\;2a_1$.
\item[$(B_A)_3$ :] $3a_3,\;a_2+2a_3,\;a_1+2a_3,\;a_1+a_2+a_3,\;2a_1+a_3,\;2a_1+a_2,\;3a_1$.
\item[$(B_A)_4$ :] $a_2+3a_3,\;a_1+3a_3,\;a_1+a_2+2a_3,\;2a_1+a_2+a_3,\;3a_1+a_3,\;3a_1+a_2$.
\item[$(B_A)_5$ :] $a_1+a_2+3a_3,\;3a_1+a_2+a_3$.
\end{enumerate}
Simultaneously, we collect the following monomials in $\mathscr{N}_1$:
\[
\mathscr{N}_1^{\prime}:=\left\{
\begin{aligned}
&\;x_2^2,\:x_2^2x_3,\;x_3^4,\;x_2^2x_3^2,\;x_1^2x_3^2,\;x_1^4,\\ 
&\;x_2^2x_3^3,\;x_1^2x_3^3,\;x_1^2x_2x_3^2,\;x_1^4x_3,\;x_1^2x_2x_3^3,\;x_1^4x_2x_3\;
\end{aligned}
\right\}.
\]
They are the corresponding monomials of the representations in $\bigcup_{t=2}^6(N_1^{\prime})_t$.
In Remark 3.18, we will observe that $\mathscr{N}_0\cap\mathscr{N}_1\subset\mathscr{N}_1^{\prime}$.

The monomials in $\mathscr{N}_2$ are obtained from the equivalence classes in $B_A$ that contain more than two elements. The followings are all the equivalence classes in $B_A$ containing more than two elements: 
\begin{enumerate}
\item[$\Gamma_1$] $:=\lbrace\;\underbrace{(11, 3, 2)}_{a_2+3a_3},\;\underbrace{(3, 3, 6)}_{2a_1+a_3}\;\rbrace$ with $(11, 3, 2)\prec(3, 3, 6)$.
\item[$\Gamma_2$] $:=\lbrace\;\underbrace{(11, 4, 5)}_{a_1+a_2+3a_3},\;\underbrace{(3, 4, 9)}_{3a_1+a_3}\;\rbrace$ with $(11, 4, 5)\prec (3, 4, 9)$.
\item[$\Gamma_3$] $:=\lbrace\;\underbrace{(8, 3, 5)}_{a_1+a_2+2a_3},\;\underbrace{(0, 3, 9)}_{3a_1}\;\rbrace$ with $(8, 3, 5)\prec (0, 3, 9)$.
\item[$\Gamma_4$] $:=\lbrace\;\underbrace{(8, 2, 2)}_{a_2+2a_3},\;\underbrace{(0, 2, 6)}_{2a_1}\;\rbrace$ with $(8, 2, 2)\prec (0, 2, 6)$.
\item[$\Gamma_5$] $:=\lbrace\;\underbrace{(9, 3, 0)}_{3a_3},\;\underbrace{(5, 3, 8)}_{2a_1+a_2+a_3}\;\rbrace$ with $(9, 3, 0)\prec (5, 3, 8)$.
\item[$\Gamma_6$] $:=\lbrace\;\underbrace{(9, 4, 3)}_{a_1+3a_3},\;\underbrace{(5, 4, 11)}_{3a_1+a_2+a_3}\;\rbrace$ with $(9, 4, 3)\prec (5, 4, 11)$.
\item[$\Gamma_7$] $:=\lbrace\;\underbrace{(6, 3, 3)}_{a_1+2a_3},\;\underbrace{(2, 3, 11)}_{3a_1+a_2}\;\rbrace$ with $(11, 4, 5)\prec (3, 4, 9)$.
\item[$\Gamma_8$] $:=\lbrace\;\underbrace{(6, 2, 0)}_{2a_3},\;\underbrace{(2, 2, 8)}_{2a_1+a_2}\;\rbrace$ with $(6, 2, 0)\prec (2, 2, 8)$.
\end{enumerate}
Then, 
\[\mathscr{N}_2=\lbrace\;x_1^2x_3y_1^2,\;x_1^3x_3y_1^2,\;x_1^3y_1^2,\;x_1^2y_1^2,\;x_1^2x_2x_3y_1,\;x_1^3x_2x_3y_1,\;x_1^3x_2y_1^2,\;x_1^2x_2y_1\;\rbrace.\]
Removing redundant elements for generating $\initial_{\prec}(\ker\pi)$ from $\mathscr{N}_1^{\prime}\cup\mathscr{N}_2$, 
\[\mathscr{N}_0=\lbrace\;x_2^2,\;x_3^4,\;x_1^2x_3^2,\;x_1^4,\;x_1^2y_1^2,\;x_1^2x_2y_1\;\rbrace.\]
Therefore,
\[\initial_{\prec}(\ker\pi)=(\;x_2^2,\;x_3^4,\;x_1^2x_3^2,\;x_1^4,\;x_1^2y_1^2,\;x_1^2x_2y_1\;),\]
and the reduced Gr\"obner basis of $\ker\pi$ is 
\[\left\{
\begin{aligned}
&\;x_2^2-y_1y_2,\;x_3^4-y_1^3y_2,\;x_1^2x_3^2-x_2y_1y_2y_3,\;\\ 
&\;x_1^4-y_2y_3^3,\;x_1^2y_1^2-x_2x_3^2y_3,\;x_1^2x_2y_1-x_3^2y_3^2\;
\end{aligned}
\right\}.\]

\begin{remark}
The subset $\mathscr{N}_1^{\prime}$ of $\mathscr{N}_1$ in the example above is, by construction, described as follows:
\[\mathscr{N}_1^{\prime}=\lbrace\;n\in\mathscr{N}_1\mid {}^{\exists}x_i, {}^{\exists}b\in B_A\text{ s.t. }n=x_im_b\text{ with }m_b\in(x_i,\ldots, x_c)\;\rbrace.\]
Then, $\mathscr{N}_1\cap\mathscr{N}_0\subset\mathscr{N}_1^{\prime}$.
\end{remark}
\begin{proof}
Let $n\in\mathscr{N}_1$ and $i:=\min\lbrace j\mid\text{$n$ is divisible by $x_j$}\rbrace$. Assume $n\notin\mathscr{N}_1^{\prime}$. Then, $n/x_i\neq m_{b}$ for any $b\in B_A$, and there exist $j>i$ and $b\in B_A$ such that $n=x_jm_b$. Then, $m_b$ is divisible by $x_i$. By Lemma 3.10, $n/x_i=x_jm_{(b-a_i)}$. Thus $n/x_i\in\mathscr{N}_1$ and $n\notin\mathscr{N}_0$.
\end{proof}

\section{\large Degree bounds via reduction numbers}
\begin{definition}
The \textit{Castelnuovo-Mumford regularity}, or simply \textit{regularity}, of a graded complex of free $S$-modules
\[\mathbf{F}:\;\cdots\to F_i\to F_{i-1}\to\cdots,\quad\text{with $F_i=\oplus_jS(-a_{i, j}),\;a_{i, j}\in\mathbb{Z}_{\geq0}$},\]
is the supremum of the numbers $a_{i, j}-i$. The regularity of a finitely generated graded $S$-module $M$, denoted by $\reg M$, is the regularity of a minimal graded free resolution of $M$. 
\end{definition}

It is noted that each finitely generated graded $S$-module has a minimal free resolution, which is unique up to graded isomorphism of complexes (see \cite[Theorem 1.6]{Eis05}). 

\begin{remark}
Let $\mathbf{F}$ be a minimal free resolution of $\ker\pi$. 
\[\mathbf{G}:\;\cdots\to G_{i+1}:=F_i\to\cdots\to G_1:=F_0\to G_0:=S,\]
is a minimal free resolution of $S/\ker\pi$. Then, $\reg\mathbb{K}[B]+1=\reg(\ker\pi)$.
\end{remark}

\begin{definition}
We define
\[r(\mathbb{K}[B]):=\max\{\deg b\mid b\in B_A\}.\]
The number is called the \textit{reduction number} of $\mathbb{K}[B]$.
\end{definition}

In general, $r(\mathbb{K}[B])\leq\reg\mathbb{K}[B]$ (see \cite[Section 2]{Nit}). Thus, 
\[r(\mathbb{K}[B])+1\leq\reg\mathbb{K}[B]+1=\reg(\ker\pi).\]

\begin{proposition}
$\mathscr{N}_1$ consists of monomials of degree at most $r(\mathbb{K}[B])+1$.
\end{proposition}
\begin{proof}
In Section 3, we defined
\[\mathscr{N}_1:=\lbrace\text{$n\mid {}^{\exists}x_i$, ${}^{\exists}b\in B_A$ s.t. $n=x_im_b$ and $n\neq m_{(a_i+b)}$}\rbrace.\]
Then, $\deg n\leq r(\mathbb{K}[B])+1$ for every $n\in\mathscr{N}_1$. 
\end{proof}

Proposition 4.4 means that $\mathscr{N}_1$ consists of monomials of degree at most $\reg(\ker\pi)$. 
The following example shows that $\mathscr{N}_2$ may contain monomials of degree greater than $\reg(\ker\pi)$.

\begin{example}
We now return to Example 3.12: Let $B$ be a simplicial affine monoid with
\[\Hilb(B)=\lbrace\;\underbrace{(11, 1)}_{a_1},\;\underbrace{(9, 3)}_{a_2},\;\underbrace{(4, 8)}_{a_3},\;\underbrace{(1, 11)}_{a_4},\;\underbrace{(12, 0)}_{e_1},\;\underbrace{(0, 12)}_{e_2}\;\rbrace.\]
In $B_A$, there exists an equivalence class 
\[\Gamma=\lbrace\;\underbrace{(31, 5)}_{2a_1+a_2},\;\underbrace{(19, 17)}_{2a_2+a_4}\;\underbrace{(7, 41)}_{a_3+3a_4}\;\rbrace,\]
with $(31, 5)\prec(19, 17)\prec(7, 41)$. Then, we have $x_3x_4^3y_1^2\in\mathscr{N}_2$.

Using MonomialAlgebras\cite{MA} via the command \texttt{regularityMA}, we know that $\reg(\ker\pi)=5$. 
It is noted that $x_3x_4^3y_1^2\notin\mathscr{N}_0$, that is, $x_3x_4^3y_1^2$ is redundant for generating $\initial_{\prec}(\ker\pi)$ since $x_3x_4^2y_1\in\mathscr{N}_2$ (see Example 3.12).
\end{example}

In Subsection 2.3, we obtained an isomorphism of $\mathbb{Z}^d$-graded $T$-modules,
\begin{equation}
\mathbb{K}[B]\simeq\oplus_{i=1}^e\tilde{I_i}(-h_i) \tag{4.1}
\end{equation}
where $\tilde{I_i}:=(y^{\frac{x-h_i}{\alpha}}\mid x\in\Gamma_i)\subset T:=\mathbb{K}[y_1,\ldots, y_d]$. 

\begin{theorem}
If, for every ideal $\tilde{I_i}$ in the decomposition (4.1) of $\mathbb{K}[B]$, either $\tilde{I_i}=T$ or $\tilde{I_i}$ is generated by monomials of degree $1$, then $\mathscr{N}_1\cup\mathscr{N}_2$ consists of monomials of degree at most $r(\mathbb{K}[B])+1$. In particular, $\initial_{\prec}(\ker\pi)$ is generated by elements of degree at most $r(\mathbb{K}[B])+1$.
\end{theorem}
\begin{proof}
Let $\Gamma=\{b_1,\ldots, b_t\}$ with $b_1\prec\cdots\prec b_t$ be an equivalence class in $B_A$. Let $\tilde{I}$ be an ideal in (4.1) corresponding to $\Gamma$ and $h:=(\min\lbrace b_{[i]}\mid b\in\Gamma\rbrace)_{i=1,\ldots, d}\;$. Then, $\tilde{I}=(y^{\frac{b_1-h}{\alpha}},\ldots,\;y^{\frac{b_t-h}{\alpha}})$. Assume that $\tilde{I}$ is generated by monomials of degree $1$, that is, $b_i-h\in\{e_1,\ldots, e_d\}$ for every $i=1,\ldots, t$. Thus, for $1\leq i<j\leq t$,
\[b_i=h+e_k,\;b_j=h+e_l\quad\text{for some $1\leq k<l\leq d$}.\]
Then, $n(b_i, b_j)=m_{b_j}y_k$ and $\deg n(b_i, b_j)\leq r(\mathbb{K}[B])+1$. 
\end{proof}

\begin{corollary}
If $\mathbb{K}[B]$ is Buchsbaum (in particular, Cohen-Macaulay), then $\initial_{\prec}(\ker\pi)$ is generated by elements of degree at most $r(\mathbb{K}[B])+1$.
\end{corollary}
\begin{proof}
It follows from the characterization in terms of the decomposition (4.1): $\mathbb{K}[B]$ is Buchsbaum if and only if either $\tilde{I_i}=T$ or $\tilde{I_i}=T_+$ and $h_i+b\in B$ for all $b\in\Hilb(B)$ \cite[Proposition 3.1(4)]{BEN}. It is also noted that $\mathbb{K}[B]$ is Cohen-Macaulay if and only if $\tilde{I_i}=T$ \cite[Proposition 3.1(2)]{BEN}.
\end{proof}

\begin{example}
Let $B$ be a simplicial affine monoid with
\[\Hilb(B)=\left\{
\begin{aligned}
\;&\underbrace{(2, 0, 1)}_{a_1},\;\underbrace{(1, 2, 0)}_{a_2},\;\underbrace{(1, 1, 1)}_{a_3},\;\underbrace{(1, 0, 2)}_{a_4},\;\underbrace{(0, 2, 1)}_{a_5},\;\underbrace{(0, 1, 2)}_{a_6},\;\\
&\underbrace{(3, 0, 0)}_{e_1},\;\underbrace{(0, 3, 0)}_{e_2}\;, \underbrace{(0, 0, 3)}_{e_3}
\end{aligned}
\right\}.\]
In \cite[Remark 2.2]{HHS}, it is shown that $y_3, y_2, y_1$ is not a generic sequence of $S/\ker\pi$ in the sense of \cite[Definition 1.5]{BaSt}. In other words, in this case, the result of Bayer and Stillman\cite[Corollary 2.5]{BaSt} does not allow us to determine whether $\initial_{\prec}(\ker\pi)$ is generated in degrees at most $\reg(\ker\pi)$.

On the other hand, the decomposition (4.1) of $\mathbb{K}[B]$ is as follows: 
\[\mathbb{K}[B]\simeq (y_2, y_3)\oplus T^{\oplus 8}\quad\text{(degree shifts omitted)}.\]
Hence, by Theorem 4.6, $\initial_{\prec}(\ker\pi)$ is generated in degrees at most $r(\mathbb{K}[B])+1$.
\end{example}

\section{\large Three-dimensional case and counterexamples}

In this section, we consider the three-dimensional case, that is,
\[\Hilb(B)=\lbrace a_1,\ldots, a_c, e_1, e_2, e_3\rbrace\subset\Bigl\{ x\in\mathbb{Z}_{\geq0}^3\mid\sum_{i=1}^3x_{[i]}=\alpha\Bigr\},\]
for some $\alpha>0$. By Proposition 4.4, every monomial in $\mathscr{N}_1$ has degree at most $\reg(\ker\pi)$. On the other hand, $\mathscr{N}_2$ may contain monomials of degree greater than $\reg(\ker\pi)$. In what follows, we fix an arbitrary $n\in\mathscr{N}_2\cap\mathscr{N}_0$. We first derive sufficient conditions for $\deg(n)\leq\reg(\ker\pi)$.

\subsection{The short resolution}
 We first recall the combinatorial description of the \textit{short resolution} constructed by Pis\'on Casares\cite{Pis}. The short resolution is a minimal free resolution of $\mathbb{K}[B]$ as a $\mathbb{Z}^3$-graded $T$-module, 
\[
0\longrightarrow F_2\xrightarrow{\Phi_2}F_{1}\xrightarrow{\Phi_1}F_{0}\xrightarrow{\Phi_0}\mathbb{K}[B]\longrightarrow 0,\]
with $F_i=\oplus_jT(-b_{i, j}),\;b_{i, j}\in B\subset\mathbb{Z}^3$, where $T:=\mathbb{K}[y_1, y_2, y_3]$. This resolution is unique up to graded isomorphism of complexes. 
 
Let $S:=\mathbb{K}[x_1, \ldots, x_c,y_1, y_2, y_3]$. To compute the regularity, we first convert the $\mathbb{Z}^3$-grading into a $\mathbb{Z}$-grading. In Subsection 2.3, we defined $\deg b:=\frac{1}{\alpha}\sum_{i=1}^3b_{[i]}$ for $b\in B\subset\mathbb{Z}^3$. This grading makes the isomorphism $S/\ker(\pi)\simeq\mathbb{K}[B]$ an isomorphism of $\mathbb{Z}$-graded $S$-modules, with the $S$-module structure on $\mathbb{K}[B]$ used throughout the paper. The $\mathbb{Z}$-graded $T$-module structure is induced by $T\hookrightarrow S\twoheadrightarrow\mathbb{K}[B]$. The regularity of $\mathbb{K}[B]$ as a $T$-module is given by
 \[\reg_{T}(\mathbb{K}[B])=\max_{i, j}\{\deg b_{i,j}-i\}.\]
 Although we define $\reg(\mathbb{K}[B])$ in terms of its minimal free resolution as an $S$-module in Section 4, we have $\reg(\mathbb{K}[B])=\reg_{T}(\mathbb{K}[B])$ (see \cite[Section 2]{Nit}).

Let $T_{+}:=(y_1, y_2, y_3)$. For each $b\in B$, 
\[W_i(b):=\bigl(\ker\Phi_{i}/T_{+}\ker\Phi_{i}\bigr)_b,\]
where $(-)_b$ denotes the degree $b$ component of a $\mathbb{Z}^3$-graded module.

Pis\'on Casares describes the short resolution in terms of certain simplicial complexes. For each $b\in B$, we consider the simplicial complex
\[T_{b}:=\Bigl\{F\subset\ \{e_1, e_2, e_3\} \mid b-\sum_{e_i\in F}e_i\in B\Bigr\}.\]
Its chain complex is given by
\[C_i(T_b):=\bigoplus_{F\in T_b,\;\# F=i+1}\mathbb{K}\cdot F,\quad i=-1, 0,1,2,\]
with differential $d_i : C_i(T_b)\to C_{i-1}(T_b)$ defined by
\[d_i(F):=\sum_{e\in F}s(e, F)\cdot (F\setminus \{e\}),\quad i=0, 1,2,\]
where $s(e, F):=(-1)^{k-1}$ if $e$ is the $k$-th element of $F$. 

Let $\tilde{H}_i(T_b)$ be the $i$-th reduced homology of $T_b$. The following result identifies the graded components $W_i(b)$ with the reduced homology of $T_b$.
\begin{proposition}[{\cite[Proposition 2.1]{Pis}}]
\[\tilde{H}_i(T_b)\simeq W_i(b),\]
for any $b\in B$ and for any $i=0,1$.
\end{proposition}
\begin{flushright}
$\square$
\end{flushright}

\subsection{Connectedness of $T_{b_1\vee b_2}$}
For each $b\in B$, by Nakayama's lemma, $\dim_{\mathbb{K}}W_0(b)$ is the number of minimal generators of $\ker\Phi_{0}$ of $\mathbb{Z}^3$-degree $b$. Thus, if $\dim_{\mathbb{K}}W_0(b)\neq 0$, then $T(-b)$ occurs as a direct summand of $F_1$ in the short resolution. Then $\deg b-1\leq\reg(\mathbb{K}[B])$, and hence $\deg b\leq\reg(\mathbb{K}[B])+1=\reg(\ker\pi)$. 

On the other hand, $\dim_{\mathbb{K}}\tilde{H}_0(T_b)$ equals the number of connected components of $T_b$ minus $1$. Combining this with Proposition 5.1, it follows that
\[\text{$T_b$ is disconnected}\implies \deg b\leq\reg(\ker\pi).\] 

Let $n\in\mathscr{N}_2\cap\mathscr{N}_0$. As we have seen in Section 3, we can write $n=n(b_1, b_2)$ with $b_1\prec b_2$ in an equivalence class $\Gamma\subset B_A$, where $A=\alpha\mathbb{Z}_{\geq 0}^3$. Then
\[\deg n=\deg (b_1\vee b_2),\]
where $b_1\vee b_2\in B$ is defined in Definition 3.3. Moreover, as we have seen in Remark 3.16, we may assume that $b_1=\min_{\prec}\{b\in\Gamma\mid b_1\vee b_2-b\in A\}$. In what follows, we derive a necessary and sufficient condition for $T_{b_1\vee b_2}$ to be connected (Theorem 5.7). 

\begin{proposition}
Let $n(b_1,b_2)\in\mathscr{N}_2\cap\mathscr{N}_0$, and let $\Gamma$ be the equivalence class containing both $b_1$ and $b_2$. If $T_{b_1\vee b_2}$ is connected, then there exists $b\in\Gamma$ such that
\begin{equation}
b_1\vee b_2-b\in A\quad\text{and}\quad b_2\precneqq b. \tag{5.1}
\end{equation}
\end{proposition}
\begin{proof}
Let $\Lambda_1\subset\{1, 2, 3\}$ be the set of indices such that $(b_1)_{[\lambda]}>(b_2)_{[\lambda]}$. Since $b_1, b_2\in B_A$, both $\Lambda_1$ and $\{1, 2, 3\}\setminus\Lambda_1$ are nonempty. The connectedness of $T_{b_1\vee b_2}$ implies that there exist $\lambda_1\in\Lambda_1$ and $\lambda_2\notin\Lambda_1$ such that $b_1\vee b_2-(e_{\lambda_1}+e_{\lambda_2})\in B$. Since $B=B_A+A$, there exists $b\in\ B_A$ such that $b_1\vee b_2-(e_{\lambda_1}+e_{\lambda_2})\in b+A$. Since $e_{\lambda_1},e_{\lambda_2}\in A$, it follows that $b_1\vee b_2-b\in A$. Moreover, since $b_1\vee b_2\in b_1+A$, it follows that $b-b_1\in\alpha\mathbb{Z}^3$. Thus $b$ is equivalent to $b_1$, and hence $b\in\Gamma$.

We next show that $b_2\precneqq b$. The vector $b_1\vee b_2-(e_{\lambda_1}+e_{\lambda_2})$ has the entries
\[\bigl(b_1\vee b_2-(e_{\lambda_1}+e_{\lambda_2})\bigr)_{[\lambda_1]}=(b_1)_{[\lambda_1]}-\alpha<(b_1)_{[\lambda_1]},\]
and
\[\bigl(b_1\vee b_2-(e_{\lambda_1}+e_{\lambda_2})\bigr)_{[\lambda_2]}=(b_2)_{[\lambda_2]}-\alpha<(b_2)_{[\lambda_2]}.\]
Thus, neither $b_1+A$ nor $b_2+A$ contain $b_1\vee b_2-(e_{\lambda_1}+e_{\lambda_2})$.  Since $b_1\vee b_2-(e_{\lambda_1}+e_{\lambda_2})\in b+A$, we have $b\neq b_1$ and $b\neq b_2$. Since $b_1\vee b_2-b\in A$, we have $(b_1\vee b_2)_{[t]}\geq(b\vee b_2)_{[t]}$ for every index $t$. Moreover, 
\[(b_1)_{[\lambda_1]}>\bigl(b_1\vee b_2-(e_{\lambda_1}+e_{\lambda_2})\bigr)_{[\lambda_1]}\geq b_{[\lambda_1]}\]
Hence $(b_1\vee b_2)_{[\lambda_1]}=(b_1)_{[\lambda_1]}>(b\vee b_2)_{[\lambda_1]}$. Then, there exists $a\in A\setminus\{0\}$ such that $b_1\vee b_2=(b\vee b_2)+a$. As we have seen in Remark 3.2, we have $m_{(b_1\vee b_2-b_2)}=m_{(b\vee b_2-b_2)}m_a$. Suppose, to the contrary, that $b\prec b_2$. Then, $n(b, b_2)=m_{b_2}m_{(b\vee b_2-b_2)}\in\mathscr{N}_2$, and 
\[n(b_1, b_2)=m_{b_2}m_{(b_1\vee b_2-b_2)}=m_{b_2}m_{(b\vee b_2-b_2)}m_a=n(b, b_2)m_a.\]
Since $m_a\neq 1$, this contradicts that $n(b_1, b_2)\in\mathscr{N}_0$.
\end{proof}

Thus, the nonexistence of an element $b\in\Gamma$ satisfying (5.1) is a sufficient condition for $T_{b_1\vee b_2}$ to be disconnected. We now consider the case where such an element $b$ exists and derive further sufficient conditions for the disconnectedness of $T_{b_1\vee b_2}$.  

\begin{lemma}
Let $n(b_1,b_2)\in\mathscr{N}_2\cap\mathscr{N}_0$, and let $\Gamma$ be the equivalence class containing both $b_1$ and $b_2$. Suppose that $b\in\Gamma$ satisfies (5.1). Let $\lambda_1$ be the index of the last nonzero entry of $b_1-b$, and let $\lambda_2$ be that of $b_2-b$.
\begin{enumerate}
\item $\lambda_2<\lambda_1$.
\item $(b_2)_{[\lambda_1]}=b_{[\lambda_1]}>(b_1)_{[\lambda_1]}$ and $(b_2)_{[\lambda]}=b_{[\lambda]}=(b_1)_{[\lambda]}$ for every $\lambda>\lambda_1$.
\end{enumerate}
\end{lemma}
\begin{proof}
Since $b_1\prec b_2\prec b$, the last nonzero entries of $b_1-b_2$, $b_1-b$, $b_2-b$ are all negative. Since $b_1-b_2=(b_1-b)-(b_2-b)$, it follows that $\lambda_2\leq\lambda_1$. If $\lambda_1=\lambda_2$, then $(b_1)_{[\lambda_1]}<b_{[\lambda_1]}$ and $(b_2)_{[\lambda_1]}<b_{[\lambda_1]}$. Hence $b_1\vee b_2-b\notin A$, contradicting (5.1). Thus $\lambda_2<\lambda_1$.  

For every $\lambda>\lambda_1$, both $(b_1-b)_{[\lambda]}$ and $(b_2-b)_{[\lambda]}$ are zero. Moreover, since $\lambda_1$ is the last nonzero entry of $b_1-b$, we have $b_{[\lambda_1]}>(b_1)_{[\lambda_1]}$. Since $\lambda_2<\lambda_1$, we have $(b_2)_{[\lambda_1]}=b_{[\lambda_1]}$.
\end{proof}

For a vector $v$, we use $+$, $-$, and $0$ to indicate that the corresponding entry of $v$ is positive, negative, and zero, respectively, while $\ast$ indicates an arbitrary entry. 

\begin{lemma}
Let $n(b_1,b_2)\in\mathscr{N}_2\cap\mathscr{N}_0$, and let $\Gamma$ be the equivalence class containing both $b_1$ and $b_2$. Suppose that $b\in\Gamma$ satisfies (5.1). Then
\begin{enumerate}
\item $b_1-b=(\ast, \ast, -)$ and $b_2-b=(+, -, 0)$. Moreover, $(b_1-b)_{[2]}\geq 0$.
\item $b_1\vee b_2-(e_2+e_3)\notin B$.
\end{enumerate}
\end{lemma}
\begin{proof}
By the previous lemma, $\lambda_2<\lambda_1$, and $(b_2-b)_{[\lambda_1]}=0$. Since $b, b_2\in\Gamma$, the vector $b_2-b$ has both positive and negative entries. Moreover, its last nonzero entry $(b_2-b)_{[\lambda_2]}$ is negative. Therefore, $\lambda_1=3$, $\lambda_2=2$, and $b_2-b=(+, -, 0)$. Since $\lambda_1=3$, we have $(b_1-b)_{[3]}<0$. Moreover, since $b_1\vee b_2-b\in A$, we have $\max\{(b_1)_{[t]}, (b_2)_{[t]}\}\geq b_{[t]}$ for every $t$, and hence $(b_1-b)_{[2]}\geq 0$.

Suppose that $b_1\vee b_2-(e_2+e_3)\in B$. Then there exists $b^{\prime}\in\Gamma$ such that $b_1\vee b_2-(e_2+e_3)\in b^{\prime}+A$. Thus, it suffices to show that $b_1\vee b_2-(e_2+e_3)\notin b^{\prime}+A$ for every $b^{\prime}\in\Gamma$. If $b_1\vee b_2-e_2\notin b^{\prime}+A$, then $b_1\vee b_2-(e_2+e_3)\notin b^{\prime}+A$. Thus it remains to consider the case where $b_1\vee b_2-e_2\in b^{\prime}+A$. Since, by the first statement of the lemma, $(b_1)_{[2]}\geq b_{[2]}>(b_2)_{[2]}$, we have $b^{\prime}_{[2]}\leq(b_1\vee b_2-e_2)_{[2]}<(b_1)_{[2]}$. Then,  $(b^{\prime}\vee b_2)_{[2]}<(b_1\vee b_2)_{[2]}$ and there exits $a\in A\setminus\{0\}$ such that $b_1\vee b_2=b^{\prime}\vee b_2+a$.
If $b^{\prime}\precneqq b_2$, then, as in the proof of Proposition 5.2, we have $n(b_1, b_2)=n(b^{\prime}, b_2)m_a$, contradicting $n(b_1, b_2)\in\mathscr{N}_0$. Hence, $b_2\precneqq b^{\prime}$ or  $b^{\prime}=b_2$. In other words, $b^{\prime}$ either satisfies (5.1) or equals $b_2$. If $b^{\prime}\neq b_2$, then $b^{\prime}$ satisfies (5.1), and, by the first statement of the lemma,
\[\bigl(b_1\vee b_2-(e_2+e_3)\bigr)_{[3]}=(b_2)_{[3]}-\alpha<(b_2)_{[3]}=b^{\prime}_{[3]}.\]
Hence, $b_1\vee b_2-(e_2+e_3)\notin b_2+A$, and $b_1\vee b_2-(e_2+e_3)\notin b^{\prime}+A$. Therefore, $b_1\vee b_2-(e_2+e_3)\notin b^{\prime}+A$ for every $b^{\prime}\in\Gamma$.

The proof also shows that, when $(b_1)_{\lambda}>(b_2)_{\lambda}$, in order to establish $b_1\vee b_2-(e_\lambda+e_\mu)\notin B$, it suffices to check that $b_1\vee b_2-(e_\lambda+e_\mu)$ does not belong to any of $b_2+A$ and $b+A$.
\end{proof}

Under the assumptions of Lemma 5.4, we first consider the case $(b_1-b)_{[2]}=0$. 

\begin{proposition}
Let $n(b_1,b_2)\in\mathscr{N}_2\cap\mathscr{N}_0$, and let $\Gamma$ be the equivalence class containing both $b_1$ and $b_2$. Suppose that $b\in\Gamma$ satisfies (5.1) and $(b_1-b)_{[2]}=0$. Then, $T_{b_1\vee b_2}$ is disconnected.
\end{proposition}
\begin{proof}
Since $b_1\vee b_2-e_i\in B$ for every $i=1, 2, 3$, the simplicial complex $T_{b_1\vee b_2}$ has three vertices $e_1, e_2, e_3$. Since $b_1-b$ has both positive and negative entries, by Lemma 5.4, it follows that $b_1-b=(+, 0, -)$ and $b_2-b=(+, -, 0)$. Since $(b_1)_{[3]}<b_{[3]}=(b_2)_{[3]}$, we have
\[\bigl(b_1\vee b_2-(e_1+e_3)\bigr)_{[3]}=(b_2)_{[3]}-\alpha<(b_2)_{[3]}=b_{[3]}.\]
 If $(b_1)_{[1]}>(b_2)_{[1]}$, then, as in the proof of Lemma 5.4, it follows that $b_1\vee b_2-(e_1+e_3)\notin B$. Hence, there is no edge between $e_1$ and $e_3$ in $T_{b_1\vee b_2}$.
 
If $(b_1)_{[1]}\leq(b_2)_{[1]}$, then
\[\bigl(b_1\vee b_2-(e_1+e_2)\bigr)_{[1]}=(b_2)_{[1]}-\alpha<(b_2)_{[1]},\]
and, since $(b_2)_{[2]}<b_{[2]}=(b_1)_{[2]}$,
\[\bigl(b_1\vee b_2-(e_1+e_2)\bigr)_{[2]}=(b_1)_{[2]}-\alpha<(b_1)_{[2]}=b_{[2]}.\]
Since $(b_1)_{[2]}>(b_2)_{[2]}$, as in the proof of Lemma 5.4, we have $b_1\vee b_2-(e_1+e_2)\notin B$, so there is no edge between $e_1$ and $e_2$ in $T_{b_1\vee b_2}$.

By Lemma 5.4, $b_1\vee b_2-(e_2+e_3)\notin B$. Hence, in both cases, $T_{b_1\vee b_2}$ has an isolated vertex. Therefore, if $(b_1-b)_{[2]}=0$, then $T_{b_1\vee b_2}$ is disconnected.
\end{proof}

Next, under the assumptions of Lemma 5.4, we consider the case $(b_1-b)_{[2]}>0$. In contrast to the case $(b_1-b)_{[2]}=0$, when $(b_1-b)_{[2]}>0$, $T_b$ can be either connected or disconnected.
\begin{proposition}
Let $n(b_1,b_2)\in\mathscr{N}_2\cap\mathscr{N}_0$, and let $\Gamma$ be the equivalence class containing both $b_1$ and $b_2$. Suppose that $b\in\Gamma$ satisfies (5.1) and $(b_1-b)_{[2]}>0$.  
\begin{enumerate}
\item Suppose $(b_1)_{[1]}>(b_2)_{[1]}$. Then, $T_{b_1\vee b_2}$ is disconnected. 
\item Suppose $(b_1)_{[1]}<(b_2)_{[1]}$. Then, $T_{b_1\vee b_2}$ is connected. 
\item Suppose $(b_1)_{[1]}=(b_2)_{[1]}$. Then, $T_{b_1\vee b_2}$ is connected if and only if there exists $b^{\prime}\in\Gamma$ such that
\begin{equation}
b_1\precneqq b^{\prime}\precneqq b_2\quad\text{and}\quad b_1\vee b_2=b^{\prime}\vee b_2.\tag{5.2}
\end{equation}
\end{enumerate}
\end{proposition}
\begin{proof}
It is noted that the simplicial complex $T_{b_1\vee b_2}$ has three vertices $e_1, e_2, e_3$. By Lemma 5.4, we have $b_1-b=(\ast, +, -)$ and $b_2-b=(+, -, 0)$. 

\vspace{0.5em}
\noindent\textbf{Case 1.} If $(b_1)_{[1]}>(b_2)_{[1]}$, then
\[\bigl(b_1\vee b_2-(e_1+e_3)\bigr)_{[3]}=(b_2)_{[3]}-\alpha<(b_2)_{[3]}=b_{[3]}.\]
As in the proof of Lemma 5.4, it follows that $b_1\vee b_2-(e_1+e_3)\notin B$. Hence, there is no edge between $e_1$ and $e_3$ in $T_{b_1\vee b_2}$. By Lemma 5.4, $e_3$ is an isolated vertex of $T_{b_1\vee b_2}$. 

 \vspace{0.5em}
\noindent\textbf{Case 2.} 
 Assume that $(b_1)_{[1]}<(b_2)_{[1]}$. Since $b_1, b_2, b\in\Gamma$, the differences of their corresponding entries are multiples of $\alpha$. Thus,
 \begin{itemize}
 \item $\bigl(b_1\vee b_2-(e_1+e_2)\bigr)_{[1]}=(b_2)_{[1]}-\alpha\in b_{[1]}+\alpha\mathbb{Z}_{\geq 0}$, since $(b_1)_{[1]}<(b_2)_{[1]}$ and $b_{[1]}<(b_2)_{[1]}$.
 \item $\bigl(b_1\vee b_2-(e_1+e_2)\bigr)_{[2]}=(b_1)_{[2]}-\alpha\in b_{[2]}+\alpha\mathbb{Z}_{\geq 0}$, since $(b_2)_{[2]}<b_{[2]}<(b_1)_{[2]}$.
 \item $\bigl(b_1\vee b_2-(e_1+e_2)\bigr)_{[3]}=(b_2)_{[3]}\in b_{[3]}+\alpha\mathbb{Z}_{\geq 0}$, since $(b_1)_{[3]}<b_{[3]}=(b_2)_{[3]}$.
 \end{itemize}
 Hence, $b_1\vee b_2-(e_1+e_2)\in b+A\subset B$. Furthermore
 \begin{itemize}
 \item $\bigl(b_1\vee b_2-(e_1+e_3)\bigr)_{[1]}=(b_2)_{[1]}-\alpha\in (b_1)_{[1]}+\alpha\mathbb{Z}_{\geq 0}$, 
 \item $\bigl(b_1\vee b_2-(e_1+e_3)\bigr)_{[2]}=(b_1)_{[2]}\in (b_1)_{[2]}+\alpha\mathbb{Z}_{\geq 0}$, 
 \item $\bigl(b_1\vee b_2-(e_1+e_3)\bigr)_{[3]}=(b_2)_{[3]}-\alpha\in (b_1)_{[3]}+\alpha\mathbb{Z}_{\geq 0}$, 
 \end{itemize}
hence, $b_1\vee b_2-(e_1+e_3)\in b_1+A\subset B$. Therefore, there is an edge between $e_1$ and $e_2$, and an edge between $e_1$ and $e_3$. Thus, $T_{b_1\vee b_2}$ is connected.

\vspace{0.5em}
\noindent\textbf{Case 3.} 
 Assume that $(b_1)_{[1]}=(b_2)_{[1]}$. Then, $b_1-b=(+, +, -)$ and $b_2-b=(+, -, 0)$. We first show that $T_{b_1\vee b_2}$ is connected if there exists $b^{\prime}\in\Gamma$ satisfying (5.2). Since $(b_1)_{[2]}>(b_2)_{[2]}$, we have 
 \[(b_1)_{[2]}=(b_1\vee b_2)_{[2]}=(b^{\prime}\vee b_2)_{[2]}=b^{\prime}_{[2]}.\]
 Since the last nonzero entry of $b^{\prime}-b_2$ is negative, we have $(b^{\prime}-b_2)_{[3]}<0$. Moreover, since $b_1\prec b^{\prime}$, we have $b_1-b^{\prime}=(+, 0, -)$. Then,
 \begin{itemize}
 \item $\bigl(b_1\vee b_2-(e_1+e_3)\bigr)_{[1]}=(b_1)_{[1]}-\alpha\in b^{\prime}_{[1]}+\alpha\mathbb{Z}_{\geq 0}$, since $b^{\prime}_{[1]}<(b_1)_{[1]}$. 
 \item $\bigl(b_1\vee b_2-(e_1+e_3)\bigr)_{[2]}=(b_1)_{[2]}\in b^{\prime}_{[2]}+\alpha\mathbb{Z}_{\geq 0}$, since $b^{\prime}_{[2]}=(b_1)_{[2]}$. 
 \item $\bigl(b_1\vee b_2-(e_1+e_3)\bigr)_{[3]}=(b_2)_{[3]}-\alpha\in b^{\prime}_{[3]}+\alpha\mathbb{Z}_{\geq 0}$, since $b^{\prime}_{[3]}<(b_2)_{[3]}$.
 \end{itemize}
 Hence, $b_1\vee b_2-(e_1+e_3)\in b^{\prime}+A\subset B$. Furthermore, 
 \begin{itemize}
 \item $\bigl(b_1\vee b_2-(e_1+e_2)\bigr)_{[1]}=(b_1)_{[1]}-\alpha\in b_{[1]}+\alpha\mathbb{Z}_{\geq 0}$, since $b_{[1]}<(b_1)_{[1]}$.
 \item $\bigl(b_1\vee b_2-(e_1+e_2)\bigr)_{[2]}=(b_1)_{[2]}-\alpha\in b_{[2]}+\alpha\mathbb{Z}_{\geq 0}$, since $b_{[2]}<(b_1)_{[2]}$.
 \item $\bigl(b_1\vee b_2-(e_1+e_2)\bigr)_{[3]}=(b_2)_{[3]}\in b_{[3]}+\alpha\mathbb{Z}_{\geq 0}$, since $b_{[3]}=(b_2)_{[3]}$.
 \end{itemize}
 Hence, $b_1\vee b_2-(e_1+e_2)\in b+A\subset B$. Therefore, there is an edge between $e_1$ and $e_3$, and an edge between $e_1$ and $e_2$. Thus, $T_{b_1\vee b_2}$ is connected.
  
 Conversely, we assume that no element satisfies (5.2). We show that $b_1\vee b_2-(e_1+e_3)\notin B$. Since $(b_1)_{[1]}=(b_2)_{[1]}$ and $(b_1)_{[3]}<(b_2)_{[3]}$, we have to check $b_1\vee b_2-(e_1+e_3)\notin b^{\prime\prime}+B$ for every $b^{\prime\prime}\in\Gamma$. Since
 \[\bigl(b_1\vee b_2-(e_1+e_3)\bigr)_{[1]}=(b_1)_{[1]}-\alpha<(b_1)_{[1]},\]
 and
 \[\bigl(b_1\vee b_2-(e_1+e_3)\bigr)_{[3]}=(b_2)_{[3]}-\alpha<(b_2)_{[3]}=b_{[3]},\]
 $b_1\vee b_2-(e_1+e_3)$ does not belong to any of $b_1+A$, $b_2+A$, and $b+A$.
 It remains to check that $b_1\vee b_2-(e_1+e_3)\notin b^{\prime\prime}+B$ for $b^{\prime\prime}\in\Gamma$ such that $b^{\prime\prime}\neq b_1, b_2$ and does not satisfy (5.1). If $b_1\vee b_2\notin b^{\prime\prime}+B$, then $b_1\vee b_2-(e_1+e_3)\notin b^{\prime\prime}+B$. Thus it remains to consider the case $b_1\vee b_2\in b^{\prime\prime}+B$. We show that such an element $b^{\prime\prime}$ does not exist. Since $b^{\prime\prime}$ does not satisfy (5.1), we have $b^{\prime\prime}\prec b_2$. Moreover, as stated at the beginning of this subsection, $b_1\prec b^{\prime\prime}$ (see Remark 3.16 for details). Then, $b_1\precneqq b^{\prime\prime}\precneqq b_2$. Since, by the assumption, $b^{\prime\prime}$ does not satisfy (5.2), we have $b^{\prime\prime}\vee b_2\neq b_1\vee b_2$. As in the proof of Proposition 5.2, we have $n(b_1, b_2)=n(b^{\prime\prime}, b_2)m_a$ for some $a\in A\setminus\{0\}$, contradicting that $n(b_1, b_2)\in\mathscr{N}_0$. 
 
Hence, $b_1\vee b_2-(e_1+e_3)\notin B$, so there is no edge between $e_1$ and $e_3$ in $T_{b_1\vee b_2}$. By Lemma 5.4, $b_1\vee b_2-(e_2+e_3)\notin B$. Therefore, $e_3$ is an isolated vertex of $T_{b_1\vee b_2}$. 
\end{proof}

\begin{theorem}
Let $n(b_1,b_2)\in\mathscr{N}_2\cap\mathscr{N}_0$, and let $\Gamma$ be the equivalence class containing both $b_1$ and $b_2$. Then, $T_{b_1\vee b_2}$ is connected if and only if there exists $b\in\Gamma$ satisfying (5.1), $(b_1-b)_{[2]}>0$, and one of the following two cases holds:
\begin{enumerate}
\item $(b_1)_{[1]}<(b_2)_{[1]}$.
\item $(b_1)_{[1]}=(b_2)_{[1]}$ and there exists $b^{\prime}\in\Gamma$ satisfying (5.2).
\end{enumerate}
\end{theorem}
\begin{proof}
The result follows from Propositions 5.2, 5.5, and 5.6.
\end{proof}

\subsection{Counterexamples}
In what follows, we give counterexamples to the question of whether $\initial_{\prec}(\ker\pi)$ is generated in degrees at most $\reg(\ker\pi)$. As shown in \cite[Lemma 2.1 and Corollary 2.3]{HHS}, the question has an affirmative answer when $\dim\mathbb{K}[B]\leq 2$. We show that this fails in dimension three by giving a counterexample.

As noted above, a minimal generator of $\initial_{\prec}(\ker\pi)$ can have degree greater than $\reg(\ker\pi)$ only if there exists an equivalence class $\Gamma$ satisfying the conditions in Theorem 5.7. However, even when such an equivalence class exists, the corresponding monomial $n(b_1,b_2)$ need not have degree greater than $\reg(\ker\pi)$. Our computational examples below suggest that affine monoids $B$ admitting an equivalence class $\Gamma$ satisfying the conditions of Theorem 5.7 are relatively rare. Moreover, even among those affine monoids, only a small proportion admit a monomial $n(b_1,b_2)$ whose degree exceeds $\reg(\ker\pi)$. To investigate this phenomenon, we performed the following computational experiment.

\begin{example}
We performed a computational experiment using Macaulay2. Let $\alpha=7$ and $c=3$. There are $\binom{9}{2}=36$ vectors in $\mathbb{Z}_{\geq 0}^3$ whose entries sum to $7$. Excluding $e_1=(7, 0, 0)$, $e_2=(0, 7, 0)$, and $e_3=(0, 0, 7)$, we can choose $a_1, a_2, a_3$ from the remaining $33$ vectors to obtain an affine monoid $B$ with 
\[\Hilb(B)=\{a_1, a_2, a_3, e_1, e_2, e_3\}.\]
Although the ordering of $a_1, a_2, a_3$ affects the morphism $S\twoheadrightarrow\mathbb{K}[B]$, we identify choices that differ only by a permutation of $a_1, a_2, a_3$. Thus, we consider $\binom{33}{3}=5456$ choices up to permutation.

Among these $5456$ affine monoids, $147$ monoids admit an equivalence class satisfying the conditions of Theorem 5.7. More precisely, $147$ monoids satisfy Condition~1, whereas none satisfy Condition~2. Among the $147$ monoids, only the following three admit a monomial $n(b_1, b_2)$ whose degree exceeds $\reg(\ker\pi)$.
\begin{enumerate}
\item $\Hilb(B)=\{a_1=(0, 4, 3),\; a_2=(1, 6, 0),\; a_3=(4, 0, 3),\; e_1,\; e_2,\; e_3\}$,
\item $\Hilb(B)=\{a_1=(1, 4, 2),\; a_2=(1, 6, 0),\; a_3=(5, 0, 2),\; e_1,\; e_2,\; e_3\}$,
\item $\Hilb(B)=\{a_1=(1, 6, 0),\; a_2=(2, 4, 1),\; a_3=(6, 0, 1),\; e_1,\; e_2,\; e_3\}$.
\end{enumerate}
For the first monoid $B$, consider the following equivalence class in $B_A$:
\[\lbrace\;b_1=\underbrace{(6, 36, 0)}_{6a_2},\;b_2=\underbrace{(20, 8, 21)}_{2a_1+5a_3},\;b=\underbrace{(13, 22, 21)}_{4a_1+a_2+3a_3}\;\rbrace.\]
This equivalence class satisfies Condition~1 of Theorem 5.7. For this equivalence class, we have $n(b_1, b_2)=x_1^2x_3^5y_2^4\in\mathscr{N}_0$. Since $\reg(\ker\pi)=10$, we obtain $\deg n(b_1, b_2)=11>\reg(\ker\pi)=10$.

For the second monoid, consider the following equivalence class in $B_A$:
\[\lbrace\;b_1=\underbrace{(6, 36, 0)}_{6a_2},\;b_2=\underbrace{(27, 8, 14)}_{2a_1+5a_3},\;b=\underbrace{(20, 22, 14)}_{4a_1+a_2+3a_3}\;\rbrace.\]
This equivalence class satisfies Condition~1 of Theorem 5.7. For this equivalence class, we have $n(b_1, b_2)=x_1^2x_3^5y_2^4\in\mathscr{N}_0$. Since $\reg(\ker\pi)=10$, we obtain $\deg n(b_1, b_2)=11>\reg(\ker\pi)=10$.

For the third monoid, consider the following equivalence class in $B_A$:
\[\lbrace\;b_1=\underbrace{(6, 36, 0)}_{6a_1},\;b_2=\underbrace{(34, 8, 7)}_{2a_2+5a_3},\;b=\underbrace{(27, 22, 7)}_{a_1+4a_2+3a_3}\;\rbrace.\]
This equivalence class satisfies Condition~1 of Theorem 5.7. For this equivalence class, we have $n(b_1, b_2)=x_2^2x_3^5y_2^4\in\mathscr{N}_0$. Since $\reg(\ker\pi)=10$, we obtain $\deg n(b_1, b_2)=11>\reg(\ker\pi)=10$.
\end{example}

\section*{Acknowledgments}
The author would like to thank Professor Yasunari Nagai for clear guidance throughout this research and for many helpful discussions.

\end{document}